\documentclass[11pt,a4paper,twoside,onecolumn]{article}
\usepackage[utf8]{inputenc}
\usepackage[english]{babel}
\usepackage{amsmath}
\usepackage{amsfonts}
\usepackage{amssymb}
\usepackage{amsthm}
\usepackage[left=3.5cm,right=2.5cm,top=3cm,bottom=3cm]{geometry}
\usepackage{graphicx}
\usepackage{enumitem}
\usepackage{cite}
\usepackage[colorlinks=true,linkcolor=black,citecolor=black]{hyperref}
\usepackage{algorithmic}
\usepackage{algorithm}

\newtheorem{theorem}{Theorem}
\newtheorem{assumption}[theorem]{Assumption}
\newtheorem{lemma}[theorem]{Lemma}
\theoremstyle{remark}
\newtheorem{remark}[theorem]{Remark}

\expandafter\let\csname equation*\endcsname\relax 
\expandafter\let\csname endequation*\endcsname\relax 

\begin{document}

\title{A globally convergent and locally quadratically\\ convergent modified B-semismooth Newton method for $\ell_1$-penalized minimization}

\author{Esther Hans \thanks{Johannes Gutenberg University Mainz, Institute of mathematics, Staudingerweg 9, D-55099 Mainz, Germany. E-Mail: hanse@uni-mainz.de, raasch@uni-mainz.de}   \and Thorsten Raasch \footnotemark[1]}

\date{\today}

\maketitle

\begin{abstract}
We consider the efficient minimization of a nonlinear, strictly convex functional with $\ell_1$-penalty term. Such minimization problems appear in a wide range of applications like Ti\-khonov regularization of (non)linear inverse problems with sparsity constraints. In (2015 \textit{Inverse Problems} \textbf{31} 025005), a globalized Bouligand-semismooth Newton method was presented for $\ell_1$-Tikhonov regularization of linear inverse problems. Nevertheless, a technical assumption on the accumulation point of the sequence of iterates was necessary to prove global convergence. Here, we generalize this method to general nonlinear problems and present a modified semismooth Newton method for which global convergence is proven without any additional requirements. Moreover, under a technical assumption, full Newton steps are eventually accepted and locally quadratic convergence is achieved. Numerical examples from image deblurring and robust regression demonstrate the performance of the method.\vspace{.4cm}
\vspace{.4cm}
\end{abstract}

\section{Introduction}
We are concerned with the efficient minimization of 
\begin{equation}\label{eq:mininfinitedim}
\min_{\mathbf u\in\ell_2}g(\mathbf u)+\sum\limits_{k=1}^\infty w_k|u_k|,
\end{equation}
where $g\colon \ell_2\to\mathbb R$ is a twice Lipschitz-continuously differentiable and strictly convex functional, $\ell_2=\ell_2(\mathbb N)$ and $\mathbf w=(w_k)_k$ is a positive weight sequence with $w_k\ge w_0>0$. 
Minimization problems of the form \eqref{eq:mininfinitedim} appear in various applications from engineering and natural sciences. A well-known example is Tikhonov regularization for inverse problems with sparsity constraints, e.g.\ medical imaging, geophysics, nondestructive testing or compressed sensing, see e.g.\ \cite{EnHaNe96,FiNoWr07,GrLo08,YaZh11}. Here, one aims to solve a possibly nonlinear ill-posed operator equation $K(\mathbf u)=\mathbf f$, $K\colon\ell_2\to\ell_2$. In practice, one has to reconstruct $\mathbf u\in\ell_2$ from noisy measurement data $\mathbf f^\delta\approx\mathbf f$. In the presence of perturbed data, regularization strategies are required for the stable computation of a numerical solution to an inverse problem \cite{EnHaNe96,ScKaHoKa12}. Applying Tikhonov regularization with sparsity constraints, one minimizes a functional consisting of a suitable discrepancy term $g \colon\ell_2\to\mathbb R$ and a sparsity promoting penalty term, see e.g.\ \cite{DaDeMo04} and the references therein. Sparsity here means the a priori assumption that the unknown solution is sparse, i.e.\ $\mathbf u$ has only few nonzero entries. As an example, in the special case of a linear discrete ill-posed operator equation $K\mathbf u=\mathbf f$, $K\colon\ell_2\to\ell_2$ linear, bounded and injective, $\mathbf f\in\ell_2$, one may choose the discrepancy term $g(\mathbf u):=\frac 12 \|K\mathbf u-\mathbf f\|_{\ell_2}^2$ \cite{EnHaNe96}. For nonlinear inverse problems like parameter identification problems, convex discrepancy terms from energy functional approaches may be considered, see e.g.\ \cite{HaQu10,Kn01,LoMaMu12,MuHaMaPi13}. Sparsity of the Tikhonov minimizer with respect to a given basis can be enforced by using the penalty term in \eqref{eq:mininfinitedim}, where the weights $w_k$ act as regularization parameters, see e.g.\ \cite{GrHaSc11, GrLo08, JiMa12} and the references therein.

In current research, sparsity-promoting regularization techniques are widely used, see e.g.\ \cite{BoBrLoMa07,DaDeMo04,GeKletal12,GrLo08,JiKhMa12,JiMa12,LoMaMu12,MiUl14,MuHaMaPi13} and the references therein. Such recovery schemes usually outperform classical Tikhonov regularization with $\ell_2$ coefficient penalties in terms of reconstruction quality if the unknown solution is sparse w.r.t.\ some basis. This is the case in many parameter identification problems for partial differential equations with piecewise smooth solutions, like electrical impedance tomography \cite{GeKletal12,JiKhMa12} or inverse heat conduction scenarios \cite{BoDaMaRa10}.

There exists a variety of approaches for the numerical minimization of \eqref{eq:mininfinitedim} in the literature. In the special case of a quadratic functional $g$, iterative soft-thresholding \cite{DaDeMo04} as well as related approaches for general functionals $g$ are well-studied, see e.g.\ \cite{BoBrLoMa07,BrLoMa09,RaTe06}. Accelerated methods and gradient-based methods introduced in \cite{BeTe09,FiNoWr07,LoMaMu12,Ne07,WrNoFi09} often gain from clever stepsize choices. Homotopy-type solvers \cite{OsPrTu00} and  alternating direction methods of multipliers \cite{YaZh11} besides many others are also state-of-the-art.

Other popular approaches for the solution of \eqref{eq:mininfinitedim} are semismooth Newton methods \cite{ChNaQi00, Ul02}. A semismooth Newton method and a quasi-Newton method for the minimization of \eqref{eq:mininfinitedim} were proposed by Muoi et al.\ in the infinite-dimensional setting \cite{MuHaMaPi13}, inspired by previous work of Herzog and Lorenz \cite{GrLo08}. If $g$ is convex and smooth, it was shown e.g.\ in \cite{CoWa05, GrLo08, MiUl14}, that $\mathbf u^*\in\ell_2$ is a minimizer of \eqref{eq:mininfinitedim} if and only if $\mathbf u^*$ is a solution to the zero-finding problem
$\mathbf F\colon\mathbb \ell_2\to\ell_2$,
\begin{equation}\label{eq:F}
\mathbf F(\mathbf u):=\mathbf u-\mathbf S_{\gamma\mathbf w}(\mathbf u-\gamma\nabla g(\mathbf u))=\mathbf 0
\end{equation}
for any fixed $\gamma>0$, where $\mathbf S_{\boldsymbol \beta}(\mathbf u):=(\operatorname{sgn}(u_k)(|u_k|-\beta_k)_+)_k$ denotes the  componentwise soft thresholding of $\mathbf u$ with respect to a positive weight sequence $\boldsymbol \beta=(\beta_k)_k$, $\nabla g$ denotes the gradient of $g$ and $x_+=\max\{x,0\}$. In \cite{MuHaMaPi13}, $\mathbf F$ from \eqref{eq:F} was shown to be Newton differentiable, i.e.\ under a suitable assumption on $g$ there exists a family of slanting functions $\mathbf G\colon\ell_2\to L(\ell_2,\ell_2)$ with
\begin{equation}\label{eq:Newtonderivative}
\lim\limits_{\mathbf h\to\mathbf 0}\frac{\|\mathbf F(\mathbf u+\mathbf h)-\mathbf F(\mathbf u)-\mathbf G(\mathbf u+\mathbf h)\mathbf h\|_{\ell_2}}{\|\mathbf h\|_{\ell_2}}=0,
\end{equation}
see also \cite{ChNaQi00,GrLo08, Ul02} for the definition of Newton derivatives.
A local semismooth Newton method was defined in \cite{MuHaMaPi13} by
\begin{align}
\mathbf G(\mathbf u^{(j)})\mathbf d^{(j)}&=-\mathbf F(\mathbf u^{(j)}),\label{eq:SSNmethod}\\
\mathbf u^{(j+1)}&=\mathbf u^{(j)}+\mathbf d^{(j)},\quad j=0,1,\ldots.\label{eq:SSNupdate}
\end{align}
with a specially chosen $\mathbf G$, cf.\ \cite{ChNaQi00, Ul02}. In \cite{MuHaMaPi13}, locally superlinear convergence was proven under suitable assumptions on the functional $g$. 

Nevertheless, the above mentioned semismooth Newton methods are only locally convergent in general. In \cite{MiUl14}, a semismooth Newton method with filter globalization was presented where semismooth Newton steps are combined with damped shrinkage steps. Another globalized semismooth Newton method was developed in \cite{HaRa15}. In loc.\ cit., inspired by \cite{HaPaRa92,ItKu09,Pa90,Qi93}, the method from \cite{GrLo08} was globalized in a finite-dimensional setting for the special case of a quadratic discrepancy term
\begin{equation}\label{eq:min_quadraticg}
\min\limits_{\mathbf u\in\mathbb R^n} \frac{1}{2}\|\mathbf K\mathbf u-\mathbf f\|_2^2+\sum\limits_{k=1}^n w_k|u_k|,
\end{equation}
where $\mathbf K\in \mathbb R^{m\times n}$ is injective and $\mathbf f\in\mathbb R^m$. In \cite{HaRa15}, $\mathbf F$ was shown to be Lipschitz continuous and directionally differentiable, i.e.\ Bouligand differentiable \cite{FaPa03I,Pa90,Sh90}. For such nonlinearities a B(ouligand)-Newton method can be defined \cite{Pa90}, replacing \eqref{eq:SSNmethod} by the generalized Newton equation
\begin{equation}\label{eq:BNewtonequation}
\mathbf F'(\mathbf u^{(j)},\mathbf d^{(j)})=-\mathbf F(\mathbf u^{(j)}).
\end{equation}
In \cite{HaRa15}, the system \eqref{eq:BNewtonequation} was shown to be equivalent to a uniquely solvable mixed linear complementarity problem \cite{CoPaSt09}. By the choice \eqref{eq:BNewtonequation}, $\mathbf d^{(j)}$ automatically is a descent direction with respect to the merit functional $\Theta\colon \mathbb R^n\to\mathbb R$,
\begin{equation}\label{eq:theta}
\Theta(\mathbf u):=\|\mathbf F(\mathbf u)\|_2^ 2,
\end{equation}
cf.\ \cite{Pa90}. Additionally, this Bouligand-Newton method can be interpreted as a semismooth Newton method with a specially chosen slanting function and is therefore called a \textit{B-semismooth Newton method}, cf.\ \cite{QiSu93}. By introducing suitable damping parameters, the method can be shown to be globally convergent under a technical assumption on the in practice unknown accumulation point $\mathbf u^*$ of the sequence of iterates, see also \cite{HaPaRa92,ItKu09,Pa90,Qi93}. Indeed, if the chosen Armijo stepsizes tend to zero, the merit functional $\Theta$ has to fulfill the condition
\begin{equation}\label{eq:convergenceconditionSSN}
\lim\limits_{(\mathbf u,\mathbf v)\to(\mathbf u^*,\mathbf u^*)}\frac{\Theta(\mathbf u)-\Theta(\mathbf v)-\Theta'(\mathbf u^*,\mathbf u-\mathbf v)}{\|\mathbf u-\mathbf v\|_2}=0
\end{equation}
at $\mathbf u^*$ to ensure global convergence.

In this work, we present a modified, globally convergent semismooth Newton method for the minimization problem \eqref{eq:mininfinitedim} in the  finite-di\-men\-sio\-nal setting 
\begin{equation}\label{eq:min}
\min\limits_{\mathbf u\in\mathbb R^n} g(\mathbf u)+\sum\limits_{k=1}^n w_k|u_k|
\end{equation}
for general (not necessarily quadratic) strictly convex functionals $g\colon \mathbb R^n\to \mathbb R$. Our work is inspired by Pang \cite{Pa91}, where a globally and locally quadratically convergent modified Bouligand-Newton method was presented for the solution of variational inequalities, nonlinear complementarity problems and nonlinear programs. We take advantage of similarities of nonlinear complementarity problems and the zero-finding problem \eqref{eq:F} to propose a modified method similar to \cite{Pa91}. Starting out from \cite{HaRa15,MuHaMaPi13}, we develop a globalized B-semismooth Newton method for general possibly nonquadratic discrepancy functionals $g$. In order to achieve global convergence without any requirements on the a priori unknown accumulation point of the iterates, inspired by \cite{Pa91}, we propose a special modification of the Newton directions $\mathbf d^{(j)}$ from \eqref{eq:BNewtonequation}, retaining the descent property w.r.t.\ $\Theta$. 
The resulting generalized Newton equation is again shown to be equivalent to a uniquely solvable mixed linear complementarity problem.
 Fortunately, in our proposed scheme, under a technical assumption, full Newton steps are accepted in the vicinity of the zero of $\mathbf F$. 
 As a consequence, under an additional regularity assumption, locally quadratic convergence is achieved. Additionally, the resulting modified method can be interpreted as a generalized Newton method proposed by Han, Pang and Rangaraj \cite{HaPaRa92}. In a neighborhood of the zero of $\mathbf F$, the modified method, under a technical assumption, coincides with the B-semismooth Newton method from \cite{HaRa15} reformulated for nonquadratic $g$. If $g$ is a quadratic functional, it was shown in \cite{HaRa15} that in a neighborhood of the zero, the B-semismooth Newton method finds the exact zero of $\mathbf F$ within finitely many iterations.

Alternatively, one may consider other globalization strategies as trust region methods or path-search methods instead of the considered line-search damping strategy, see e.g.\ \cite{Ul02,FaPa03II} and the references therein. The path-search globalization strategy proposed by the authors of \cite{Ra94,DiFe95} could be a promising, albeit conceptually different, alternative. These approaches go beyond the scope of this paper and are part of future work.

For the rest of the paper, we require the following assumption on the smoothness of $g$, similar to \cite[Assumption 3.1, Example 3.4]{MuHaMaPi13}. In Section \ref{sec:globalconvergence}, we will need a further assumption regarding the locally quadratic convergence of the method.
\begin{assumption}\label{ass:1}
\begin{enumerate}[label=(A\arabic*)]
\item \label{ass:A2} The function $g$ is twice Lipschitz-continuously differentiable and the Hessian $\nabla^2g(\mathbf u)$ is positive definite for all $\mathbf u\in\mathbb R^n$. Moreover, there exist constants $0<c_1, c_2< \infty$ with 
\begin{equation}\nonumber\label{eq:hessianbounded}
c_1\|\mathbf h\|_2^2\le \langle \nabla^2 g(\mathbf u)\mathbf h,\mathbf h\rangle\le c_2\|\mathbf h\|_2^2,\qquad \text{for all } \mathbf h\in\mathbb R^n,
\end{equation}
uniformly for all $\mathbf u\in\mathbb R^n$.
\item \label{ass:A3}The level sets $L_{\Theta}(\mathbf u^{(0)}):=\{\mathbf u\in\mathbb R^n: \Theta(\mathbf u)\le \Theta(\mathbf u^{(0)})\}$ of $\Theta$ are compact.
\end{enumerate}
\end{assumption}

The compactness of the level sets in the case of a quadratic functional $g(\mathbf u)=\frac{1}{2}\|\mathbf K\mathbf u- \mathbf f\|_2^2$, $\mathbf K\in\mathbb R^{m\times n}$ injective, $n\le m$, $\mathbf f\in\mathbb R^m$ was shown in \cite{HaRa15}. Note that the positive definiteness of the Hessian $\nabla^2g(\mathbf u)$ implies strict convexity of the functional $g$ and ensures unique solvability of \eqref{eq:min}.

The paper is organized as follows. Section \ref{sec:feasibility} treats the proposed B-se\-mi\-smooth Newton method and its modification as well as their feasibility. Section \ref{sec:globalconvergence} addresses the global convergence and the local convergence speed of the methods. Numerical examples demonstrate the performance of the proposed algorithms in Section \ref{sec:numericalexperiments}.

\section{A B-semismooth Newton method and its modification}\label{sec:feasibility}

In this section, we present the algorithm of the B-semismooth Newton method from \cite{HaRa15} generalized to the minimization problem \eqref{eq:min} as well as a modified version and discuss their feasibility. Additionally, we suggest a hybrid method. We start with the modified algorithm because the generalized B-semismooth Newton method can immediately be deduced from the modified method.

\subsection{A modified B-semismooth Newton method and its feasibility}\label{sec:21}

In the following, we introduce a modified B-semismooth Newton method for the solution of \eqref{eq:min}. 
 We denote the  \textit{active set} by $\mathcal A(\mathbf u):=\mathcal A^+(\mathbf u)\cup\mathcal A^-(\mathbf u)$, where
\begin{align}
\mathcal A^+(\mathbf u)&:=\{k: \gamma(\nabla g(\mathbf u))_k+\gamma w_k<u_k\},\label{eq:A+}\\
\mathcal A^-(\mathbf u)&:=\{k: u_k<\gamma(\nabla g(\mathbf u))_k-\gamma w_k\},\label{eq:A-}
\end{align}
and the \textit{inactive set} by $\mathcal I(\mathbf u):=\mathcal I^\circ(\mathbf u)\cup\mathcal I^+(\mathbf u)\cup\mathcal I^-(\mathbf u)$, where
\begin{align}
\mathcal I^\circ(\mathbf u)&:=\{k: \gamma(\nabla g(\mathbf u))_k-\gamma w_k<u_k<\gamma(\nabla g(\mathbf u))_k+\gamma w_k\},\label{eq:I0}\\
\mathcal I^+(\mathbf u)&:=\{k: u_k=\gamma(\nabla g(\mathbf u))_k+\gamma w_k\},\label{eq:Iplus}\\
\mathcal I^-(\mathbf u)&:=\{k: u_k=\gamma(\nabla g(\mathbf u))_k-\gamma w_k\}.\label{eq:I-}
\end{align}
Below, we drop the argument $\mathbf u$ if there is no risk of confusion.

For $\mathbf F\colon\mathbb R^n\to\mathbb R^n$ defined by \eqref{eq:F}, we then have
\begin{align}
F_k(\mathbf u)&=\min\{\gamma(\nabla g(\mathbf u))_k+\gamma w_k, u_k\}, & k\in\mathcal A^+(\mathbf u)\cup\mathcal I^\circ(\mathbf u)\cup\mathcal I^+(\mathbf u),\label{eq:Fmin}\\
F_k(\mathbf u)&=\max\{\gamma(\nabla g(\mathbf u))_k-\gamma w_k, u_k\}, & k\in\mathcal A^-(\mathbf u)\cup\mathcal I^\circ(\mathbf u)\cup\mathcal I^-(\mathbf u).
\end{align}
By Assumption \ref{ass:1}, $\mathbf F$ is Lipschitz continuous and directionally differentiable. The directional de\-ri\-va\-tive of $\mathbf F$ can be easily deduced.
\begin{lemma}\label{lemma:dirderivative}
The directional derivative of $\mathbf F$ at $\mathbf u\in\mathbb R^n$ in the direction $\mathbf d\in\mathbb R^n$ is given elementwise by
\begin{equation}
F'_k(\mathbf u,\mathbf d)=\begin{cases}
\gamma(\nabla^2g(\mathbf u)\mathbf d)_k, & k\in\mathcal A(\mathbf u),\\
d_k, & k\in\mathcal I^\circ(\mathbf u),\\
\min\{\gamma(\nabla^2g(\mathbf u)\mathbf d)_k,d_k\}, & k\in\mathcal I^+(\mathbf u),\\
\max\{\gamma(\nabla^2g(\mathbf u)\mathbf d)_k,d_k\}, & k\in\mathcal I^-(\mathbf u).
\end{cases}
\end{equation}
\end{lemma}
\begin{proof}
The claim is trivially true for $k\in\mathcal A(\mathbf u)$ and $k\in\mathcal I^\circ(\mathbf u)$. For $k\in\mathcal I^+(\mathbf u)$ we have with \eqref{eq:Iplus} and \eqref{eq:Fmin}
\begin{align*}
&\lim\limits_{t\searrow 0}\frac{\min\{\gamma(\nabla g(\mathbf u+t\mathbf d))_k+\gamma w_k, u_k+td_k\}-\min\{\gamma(\nabla g(\mathbf u))_k+\gamma w_k, u_k\}}{t}\\
=&\lim\limits_{t\searrow 0}\frac{\min\{\gamma(\nabla g(\mathbf u+t\mathbf d))_k-\gamma(\nabla g(\mathbf u))_k,t d_k\}}{t}\\
=& \min\Big\{\lim\limits_{t\searrow 0}\frac{\gamma(\nabla g(\mathbf u+t\mathbf d))_k-\gamma(\nabla g(\mathbf u))_k}{t},d_k\Big\}\\
=&\min\{\gamma(\nabla^2 g(\mathbf u)\mathbf d)_k, d_k\}.
\end{align*}
The claim for $k\in\mathcal I^-(\mathbf u)$ results analogously.
\end{proof}

The directional derivative of the merit functional $\Theta$ from \eqref{eq:theta} at $\mathbf u\in\mathbb R^n$ in the direction $\mathbf d\in\mathbb R^n$ is given by $\Theta'(\mathbf u,\mathbf d)=2\langle \mathbf F'(\mathbf u,\mathbf d),\mathbf F(\mathbf u)\rangle$, where $\langle\cdot , \cdot\rangle$ denotes the Euclidean scalar product, see e.g.\ \cite[Lemma 3.2]{HaRa15}.

To introduce the modified semismooth Newton method, we define the subsets
\begin{align}
\mathcal A^+_+(\mathbf u)&:=\{k: \gamma(\nabla g(\mathbf u))_k+\gamma w_k<u_k<0\},\label{eq:A++}\\
\mathcal A^-_-(\mathbf u)&:=\{k: 0<u_k<\gamma(\nabla g(\mathbf u))_k-\gamma w_k\},\label{eq:A--}\\
\mathcal I^\circ_+(\mathbf u)&:=\{k: \gamma(\nabla g(\mathbf u))_k-\gamma w_k<u_k<\gamma(\nabla g(\mathbf u))_k+\gamma w_k<0\},\label{eq:I0+}\\
\mathcal I^\circ_-(\mathbf u)&:=\{k: 0<\gamma(\nabla g(\mathbf u))_k-\gamma w_k<u_k<\gamma(\nabla g(\mathbf u))_k+\gamma w_k\}.\label{eq:I0-}
\end{align}
Inspired by \cite{Pa91}, we define the modified index sets
\begin{align}
\overline{\mathcal A^+}(\mathbf u)&:=\mathcal A^+(\mathbf u)\setminus\mathcal A^+_+(\mathbf u),\label{eq:A+modified}\\
\overline{\mathcal A^-}(\mathbf u)&:=\mathcal A^-(\mathbf u)\setminus\mathcal A^-_-(\mathbf u),\label{eq:A-modified}\\
\overline{\mathcal I^\circ}(\mathbf u)&:=\mathcal I^\circ(\mathbf u)\setminus\big(\mathcal I^\circ_+(\mathbf u)\cup\mathcal I^\circ_-(\mathbf u)\big),\label{eq:I0modified}\\
\overline{\mathcal I^+}(\mathbf u)&:=\mathcal I^+(\mathbf u)\cup\mathcal A^+_+(\mathbf u)\cup\mathcal I^\circ_+(\mathbf u),\label{eq:I+modified}\\
\overline{\mathcal I^-}(\mathbf u)&:=\mathcal I^-(\mathbf u)\cup\mathcal A^-_-(\mathbf u)\cup\mathcal I^\circ_-(\mathbf u).\label{eq:I-modified}
\end{align}
We denote $\overline{\mathcal A}(\mathbf u):=\overline{\mathcal A^+}(\mathbf u)\cup\overline{\mathcal A^-}(\mathbf u)$ and $\overline{\mathcal I}(\mathbf u):=\overline{\mathcal I^\circ}(\mathbf u)\cup\overline{\mathcal I^+}(\mathbf u)\cup\overline{\mathcal I^-}(\mathbf u)$ respectively. The subsets \eqref{eq:A++}--\eqref{eq:I0-} fulfill $\mathcal A^+_+(\mathbf u)=\emptyset$, $\mathcal A^-_-(\mathbf u)=\emptyset$, $\mathcal I^\circ_+(\mathbf u)=\emptyset$ and $\mathcal I^\circ_-(\mathbf u)=\emptyset$ if $\mathbf F(\mathbf u)=\mathbf 0$.

In the following lemma, we consider a linear complementarity problem which is important for all further discussions, cf.\ \cite{HaRa15}.

\begin{lemma}
Let $\mathbf u\in \mathbb R^n$ and $\mathbf M:=\nabla^2 g(\mathbf u)$. The linear complementarity problem 
\begin{align}\label{eq:LCP}
\mathbf x\ge\mathbf 0,\quad \mathbf N\mathbf x+\mathbf z\ge\mathbf 0,\quad\langle\mathbf x,\mathbf N\mathbf x+\mathbf z\rangle=0,
\end{align}
with 
\begin{align}\label{eq:BSSNN}
\mathbf N=&\mathbf N(\mathbf u)\nonumber\\
:=&\gamma
\begin{pmatrix}
\mathbf M_{\overline{\mathcal I^+},\overline{\mathcal I^+}}-\mathbf M_{\overline{\mathcal I^+},\overline{\mathcal A}}\mathbf M_{\overline{\mathcal A},\overline{\mathcal A}}^{-1}\mathbf M_{\overline{\mathcal A},\overline{\mathcal I^+}}&
\mathbf M_{\overline{\mathcal I^+},\overline{\mathcal A}}\mathbf M_{\overline{\mathcal A},\overline{\mathcal A}}^{-1}\mathbf M_{\overline{\mathcal A},\overline{\mathcal I^-}}-
\mathbf M_{\overline{\mathcal I^+},\overline{\mathcal I^-}}\\
\mathbf M_{\overline{\mathcal I^-},\overline{\mathcal A}}\mathbf M_{\overline{\mathcal A},\overline{\mathcal A}}^{-1}\mathbf M_{\overline{\mathcal A},\overline{\mathcal I^+}}-\mathbf M_{\overline{\mathcal I^-},\overline{\mathcal I^+}}&
\mathbf M_{\overline{\mathcal I^-},\overline{\mathcal I^-}}-\mathbf M_{\overline{\mathcal I^-},\overline{\mathcal A}}\mathbf M_{\overline{\mathcal A},\overline{\mathcal A}}^{-1}\mathbf M_{\overline{\mathcal A},\overline{\mathcal I^-}}
\end{pmatrix}
\end{align}
and 
\begin{align}\label{eq:BSSNz}
\begin{split}
&\mathbf z=\mathbf z(\mathbf u)\\
&:=
\begin{pmatrix}
\gamma(\mathbf M_{\overline{\mathcal I^+},\overline{\mathcal A}}\mathbf M_{\overline{\mathcal A},\overline{\mathcal A}}^{-1}\mathbf M_{\overline{\mathcal A},\overline{\mathcal I^\circ}}-\mathbf M_{\overline{\mathcal I^+},\overline{\mathcal I^\circ}})\mathbf u_{\overline{\mathcal I^\circ}}
-\mathbf M_{\overline{\mathcal I^+},\overline{\mathcal A}}\mathbf M_{\overline{\mathcal A},\overline{\mathcal A}}^{-1}\mathbf F(\mathbf u)_{\overline{\mathcal A}}
+\mathbf F(\mathbf u)_{\overline{\mathcal I^+}}\\
\gamma(\mathbf M_{\overline{\mathcal I^-},\overline{\mathcal I^\circ}}-\mathbf M_{\overline{\mathcal I^-},\overline{\mathcal A}}\mathbf M_{\overline{\mathcal A},\overline{\mathcal A}}^{-1}\mathbf M_{\overline{\mathcal A},\overline{\mathcal I^\circ}})\mathbf u_{\overline{\mathcal I^\circ}}
+\mathbf M_{\overline{\mathcal I^-},\overline{\mathcal A}}\mathbf M_{\overline{\mathcal A},\overline{\mathcal A}}^{-1}\mathbf F(\mathbf u)_{\overline{\mathcal A}}
-\mathbf F(\mathbf u)_{\overline{\mathcal I^-}}
\end{pmatrix}\\
&\phantom{:=}\
-\mathbf N(\mathbf u)
\begin{pmatrix}
\mathbf u_{\overline{\mathcal I^+}}\\
-\mathbf u_{\overline{\mathcal I^-}}
\end{pmatrix}
\end{split}
\end{align}
has a unique solution.
\end{lemma}
\begin{proof}
By Assumption \ref{ass:1}, $\mathbf M=\nabla^2 g(\mathbf u)$ is symmetric and positive definite. Therefore, $\mathbf N$ from \eqref{eq:BSSNN} is symmetric and positive definite, see \cite[Lemma 3.3]{HaRa15}. Hence \eqref{eq:LCP} is uniquely solvable, see \cite[Theorem 3.3.7]{CoPaSt09} and \cite[Theorem 3.5]{HaRa15}.
\end{proof}

Now we can define the generalized Newton equation for $\mathbf F$, cf.\ \cite{HaRa15}. Let $\mathbf u\in\mathbb R^n$ and
\begin{align}\label{eq:BC}
\begin{split}
\mathcal B&=\mathcal B(\mathbf u):=\overline{\mathcal A}(\mathbf u)\cup\{k\in\overline{\mathcal I^+}(\mathbf u)\cup\overline{\mathcal I^-}(\mathbf u): x_k>0\},\\
\mathcal C&=\mathcal C(\mathbf u):=\overline{\mathcal I^\circ}(\mathbf u)\cup\{k\in\overline{\mathcal I^+}(\mathbf u)\cup\overline{\mathcal I^-}(\mathbf u): x_k=0\},
\end{split}
\end{align}
where $\mathbf x=(x_k)_k$ is the unique solution to the linear complementarity problem
\eqref{eq:LCP}. Then, by defining the generalized derivative blockwise
\begin{equation}\label{eq:G}
\begin{pmatrix}\mathbf G(\mathbf u)_{\mathcal B,\mathcal B} & \mathbf G(\mathbf u)_{\mathcal B,\mathcal C}\\ \mathbf G(\mathbf u)_{\mathcal C,\mathcal B} & \mathbf G(\mathbf u)_{\mathcal C,\mathcal C}\end{pmatrix}:=\begin{pmatrix}
\gamma (\nabla^2 g(\mathbf u))_{\mathcal B,\mathcal B} & \gamma (\nabla^2 g(\mathbf u))_{\mathcal B,\mathcal C}\\
\mathbf 0_{\mathcal C,\mathcal B} & \mathbf I_{\mathcal C,\mathcal C} 
\end{pmatrix}\in\mathbb R^{n\times n},
\end{equation}
the modified semismooth Newton method is given by
\begin{align}
\mathbf G(\mathbf u^{(j)})\mathbf d^{(j)}&=-\mathbf F(\mathbf u^{(j)}),\label{eq:generalizedNewtoneq} \\
\mathbf u^{(j+1)}&=\mathbf u^{(j)}+t_j\mathbf d^{(j)},\quad j=0,1,\ldots \label{eq:Newtonupdate}
\end{align}
 with suitably chosen damping parameters $t_j\in (0,1]$.
 
\begin{remark}
 In \cite{MuHaMaPi13}, Muoi et al.\ chose the slanting function
\begin{equation}\label{eq:GMuoi}
\begin{pmatrix}\mathbf G(\mathbf u)_{\mathcal A,\mathcal A} & \mathbf G(\mathbf u)_{\mathcal A,\mathcal I}\\ \mathbf G(\mathbf u)_{\mathcal I,\mathcal A} & \mathbf G(\mathbf u)_{\mathcal I,\mathcal I}\end{pmatrix}=\begin{pmatrix}
\gamma(\nabla^2 g(\mathbf u))_{\mathcal A,\mathcal A} & \gamma(\nabla^2 g(\mathbf u))_{\mathcal A,\mathcal I}\\
\mathbf 0_{\mathcal I,\mathcal A} & \mathbf I_{\mathcal I,\mathcal I}
\end{pmatrix},
\end{equation}
blocked according to the active and inactive sets, to define the local semismooth Newton method \eqref{eq:SSNmethod},\eqref{eq:SSNupdate}. The key difference of \eqref{eq:G} compared to \eqref{eq:GMuoi} is the modification of the index sets. Note that $\mathbf G$ from \eqref{eq:G} is not a slanting function in general because in regions where $\mathbf F$ is smooth, $\mathbf G$ does not coincide with the Fr\'{e}chet-derivative of $\mathbf F$.
\end{remark} 
  
Let $\mathbf u^{(j)}\in\mathbb R^n$ and $\mathbf M:=\nabla^2 g(\mathbf u^{(j)})$. Then $\mathbf d^{(j)}\in\mathbb R^n$ solves \eqref{eq:generalizedNewtoneq} if and only if
\begin{align}
\gamma(\mathbf M\mathbf d^{(j)})_{\overline{\mathcal A}}&=-\mathbf F(\mathbf u^{(j)})_{\overline{\mathcal A}},\label{eq:dA}\\
\mathbf d_{\overline{\mathcal I^\circ}}^{(j)}&=-\mathbf u_{\overline{\mathcal I^\circ}}^{(j)},\label{eq:dI}
\end{align}
and
\begin{equation}\label{eq:dx}
\begin{aligned}
\mathbf x&:=\begin{pmatrix}
\mathbf d_{\overline{\mathcal I^+}}^{(j)}+\mathbf u_{\overline{\mathcal I^+}}^{(j)}\\
-\mathbf d_{\overline{\mathcal I^-}}^{(j)}-\mathbf u_{\overline{\mathcal I^-}}^{(j)}
\end{pmatrix}\\
\mathbf y&:=\mathbf N(\mathbf u^{(j)})\mathbf x+\mathbf z(\mathbf u^{(j)})=\begin{pmatrix}
\gamma(\mathbf M\mathbf d^{(j)})_{\overline{\mathcal I^+}}+(\mathbf F(\mathbf u^{(j)}))_{\overline{\mathcal I^+}}\\
-\gamma(\mathbf M\mathbf d^{(j)})_{\overline{\mathcal I^-}}-(\mathbf F(\mathbf u^{(j)}))_{\overline{\mathcal I^-}}
\end{pmatrix},
\end{aligned}
\end{equation}
where $\mathbf x$, $\mathbf y$ solve the linear complementarity problem \eqref{eq:LCP}, cf.\ \cite[Lemma 3.4]{HaRa15}.

We summarize the above observations in the following lemma, cf.\ \cite[Theorem 3.5]{HaRa15}.

\begin{lemma}\label{lemma:dunique}
Let $\mathbf x$ be the unique solution to \eqref{eq:LCP} for an iterate $\mathbf u^{(j)}\in\mathbb R^n$ and $\mathbf M:=\nabla^2 g(\mathbf u^{(j)})$. Then, the Newton update $\mathbf d^{(j)}$ from \eqref{eq:generalizedNewtoneq} is given by
\begin{equation}
\begin{aligned}\label{eq:d}
\mathbf d^{(j)}_{\overline{\mathcal I^\circ}}&=-\mathbf u^{(j)}_{\overline{\mathcal I^\circ}},\\
\mathbf d^{(j)}_{\overline{\mathcal I^+}}&=\mathbf x_{\overline{\mathcal I^+}}-\mathbf u^{(j)}_{\overline{\mathcal I^+}},\\
\mathbf d^{(j)}_{\overline{\mathcal I^-}}&=-\mathbf x_{\overline{\mathcal I^-}}-\mathbf u^{(j)}_{\overline{\mathcal I^-}},\\
\mathbf d^{(j)}_{\overline{\mathcal A}}&=\frac{1}{\gamma}\mathbf M_{\overline{\mathcal A},\overline{\mathcal A}}^{-1}\big(-\gamma\mathbf M_{\overline{\mathcal A},\overline{\mathcal I}}\mathbf d^{(j)}_{\overline{\mathcal I}}-\mathbf F(\mathbf u^{(j)})_{\overline{\mathcal A}}\big).
\end{aligned}
\end{equation}
\end{lemma}

Before proceeding, we prove some useful identities similar to \cite[Lemma 2]{Pa91}.

\begin{lemma}
Let $\mathbf u\in\mathbb R^n$, $\mathbf d=\mathbf d(\mathbf u)$ the unique solution to \eqref{eq:d} and $\mathbf M=\nabla^2 g(\mathbf u)$. For $k\in\{1,\ldots,n\}$, we have the following identities
\begin{align}
(\gamma(\nabla g(\mathbf u))_k+\gamma w_k)\gamma (\mathbf M\mathbf d)_k
&=-(\gamma (\nabla g(\mathbf u))_k+\gamma w_k)^2, & k\in\overline{\mathcal A^+}(\mathbf u),\label{eq:help1}\\
(\gamma(\nabla g(\mathbf u))_k-\gamma w_k)\gamma (\mathbf M\mathbf d)_k
&=-(\gamma (\nabla g(\mathbf u))_k-\gamma w_k)^2, & k\in\overline{\mathcal A^-}(\mathbf u),\label{eq:help2}\\
u_k d_k &=-u_k^2, &k\in\overline{\mathcal I^\circ}(\mathbf u).\label{eq:help3}
\end{align}
Additionally, for $k\in \mathcal A^+_+(\mathbf u)\cup\mathcal I^\circ_+(\mathbf u)\cup\{k\in\mathcal I^+(\mathbf u): F_k(\mathbf u)<0)\}$ the inequality
\begin{align}
(\gamma(\nabla g(\mathbf u))_k+\gamma w_k)\gamma (\mathbf M\mathbf d)_k
&\le-(\gamma (\nabla g(\mathbf u))_k+\gamma w_k)^2 \label{eq:help4}
\end{align}
holds, for $k\in \mathcal A^-_-(\mathbf u)\cup\mathcal I^\circ_-(\mathbf u)\cup\{k\in\mathcal I^-(\mathbf u): F_k(\mathbf u)>0)\}$ we have 
\begin{align}
(\gamma(\nabla g(\mathbf u))_k-\gamma w_k)\gamma (\mathbf M\mathbf d)_k
&\le-(\gamma (\nabla g(\mathbf u))_k-\gamma w_k)^2\label{eq:help5}
\end{align}
and for $k\in \mathcal A^+_+(\mathbf u)\cup\mathcal A^-_-(\mathbf u)\cup\mathcal I^\circ_+(\mathbf u)\cup\mathcal I^\circ_-(\mathbf u)\cup \{k\in\mathcal I^+(\mathbf u): F_k(\mathbf u)<0)\}\cup \{k\in\mathcal I^-(\mathbf u): F_k(\mathbf u)>0)\}$ we have
\begin{align}
u_k d_k &\le-u_k^2.\label{eq:help6}
\end{align}
\end{lemma}
\begin{proof}
Equations \eqref{eq:help1}, \eqref{eq:help2} and \eqref{eq:help3} immediately follow from \eqref{eq:dA} and \eqref{eq:dI}. For $k\in \mathcal A^+_+(\mathbf u)\cup\mathcal I^\circ_+(\mathbf u)\cup\{k\in\mathcal I^+(\mathbf u): F_k(\mathbf u)<0)\}$, we have by definition $\gamma(\nabla g(\mathbf u))_k+\gamma w_k<0$ and with \eqref{eq:LCP} and \eqref{eq:dx} we have $\gamma(\mathbf M\mathbf d)_k\ge -F_k(\mathbf u)\ge -(\gamma(\nabla g(\mathbf u))_k+\gamma w_k)$ implying \eqref{eq:help4}. For $k\in \mathcal A^-_-(\mathbf u)\cup\mathcal I^\circ_-(\mathbf u)\cup\{k\in\mathcal I^-(\mathbf u): F_k(\mathbf u)>0)\}$, we have by definition $\gamma(\nabla g(\mathbf u))_k-\gamma w_k>0$ and with \eqref{eq:LCP} and \eqref{eq:dx} we have $\gamma (\mathbf M\mathbf d)_k\le -F_k(\mathbf u)\le -(\gamma(\nabla g(\mathbf u))_k-\gamma w_k)$, implying \eqref{eq:help5}.

For $k\in\mathcal A^+_+(\mathbf u)\cup\mathcal I^\circ_+(\mathbf u)\cup\{k\in\mathcal I^+(\mathbf u): F_k(\mathbf u)<0\}$, we have $u_k<0$ and $d_k\ge -u_k$ because of \eqref{eq:LCP} and \eqref{eq:dx}. If $k\in\mathcal A^-_-(\mathbf u)\cup\mathcal I^\circ_-(\mathbf u)\cup\{k\in\mathcal I^-(\mathbf u): F_k(\mathbf u)>0\}$, we have $u_k>0$ and $d_k\le -u_k$ because of \eqref{eq:LCP} and \eqref{eq:dx}. In both cases \eqref{eq:help6} follows. 
\end{proof}

Now we verify that $\mathbf d=\mathbf d(\mathbf u)$ from \eqref{eq:d} is a descent direction of the merit functional $\Theta$ from \eqref{eq:theta} at $\mathbf u$.

\begin{lemma}\label{lemma:descent}
Let $\mathbf u\in\mathbb R^n$ with $\Theta(\mathbf u)>0$ and $\mathbf M:=\nabla^2 g(\mathbf u)$. Let $\mathbf d=\mathbf d(\mathbf u)\in\mathbb R^n$ be the solution to \eqref{eq:d}. Then, we have 
\begin{align}\label{eq:descent}
\Theta'(\mathbf u,\mathbf d)\le -2\Theta(\mathbf u)<0,
\end{align}
i.e.\ $\mathbf d$ is a true descent direction of $\Theta$ at $\mathbf u$ in the  direction $\mathbf d$.
\end{lemma}
\begin{proof}
The proof follows the idea of \cite[Proof of Proposition 5]{Pa91}. We have 
\begin{equation}\nonumber
\Theta'(\mathbf u,\mathbf d)=2\langle\mathbf F(\mathbf u),\mathbf F'(\mathbf u,\mathbf d)\rangle=2\sum\limits_{i=1}^8 T_i,
\end{equation}
where we have with Lemma \ref{lemma:dirderivative}
\begin{align*}
\begin{split}
T_1&:=\sum\limits_{k\in\overline{\mathcal A}(\mathbf u)} F_k(\mathbf u)\gamma(\mathbf M\mathbf d)_k,\\
T_3&:=\sum\limits_{k\in\mathcal I^+(\mathbf u)} F_k(\mathbf u)\min\{d_k,\gamma(\mathbf M\mathbf d)_k\},\\
T_5&:=\sum\limits_{k\in\mathcal A^+_+(\mathbf u)} F_k(\mathbf u)\gamma(\mathbf M\mathbf d)_k,\\
T_7&:=\sum\limits_{k\in\mathcal I^\circ_+(\mathbf u)} u_k d_k,\\
\end{split}
\hspace{-0.27cm}
\begin{split}
T_2&:=\sum\limits_{k\in\overline{\mathcal I^\circ}(\mathbf u)} u_k d_k,\\
T_4&:=\sum\limits_{k\in\mathcal I^-(\mathbf u)} F_k(\mathbf u)\max\{d_k,\gamma(\mathbf M\mathbf d)_k\},\\
T_6&:=\sum\limits_{k\in\mathcal A^-_-(\mathbf u)} F_k(\mathbf u)\gamma(\mathbf M\mathbf d)_k,\\
T_8&:=\sum\limits_{k\in\mathcal I^\circ_-(\mathbf u)} u_k d_k.
\end{split}
\end{align*}

For $k\in\mathcal I^+(\mathbf u)$, we have 
\begin{align*}
\min\{d_k,\gamma(\mathbf M\mathbf d)_k\}=\min\{d_k+F_k(\mathbf u),\gamma(\mathbf M\mathbf d)_k+F_k(\mathbf u)\}-F_k(\mathbf u)=-F_k(\mathbf u)
\end{align*}
because of \eqref{eq:LCP} and \eqref{eq:dx}. Similarly, for $k\in\mathcal I^-(\mathbf u)$ we have
\begin{align*}
\max\{d_k,\gamma(\mathbf M\mathbf d)_k\}=\max\{d_k+F_k(\mathbf u),\gamma(\mathbf M\mathbf d)_k+F_k(\mathbf u)\}-F_k(\mathbf u)=-F_k(\mathbf u).
\end{align*}
With \eqref{eq:help1}--\eqref{eq:help6}, we obtain
\begin{equation}\nonumber
\Theta'(\mathbf u,\mathbf d)=2\sum\limits_{i=1}^8 T_i\le -2\sum\limits_{k=1}^n F_k(\mathbf u)^2=-2\Theta(\mathbf u),
\end{equation}
finishing the proof.
 \end{proof}

We choose the stepsizes $t_j\in (0,1]$ in \eqref{eq:Newtonupdate} by the well-known Armijo rule
\begin{equation}\nonumber
t_j:=\max\{\beta^l: \Theta(\mathbf u^{(j)}+\beta^{l}\mathbf d^{(j)})\le (1-2\sigma \beta^l)\Theta(\mathbf u^{(j)}),\quad l=0,1,\ldots\},
\end{equation}
where $\beta\in (0,1)$ and $\sigma\in(0,\frac{1}{2})$, see also \cite{HaPaRa92, HaRa15,ItKu09, Pa90, Pa91, Qi93}. These stepsizes can be computed in finitely many iterations. We cite the following lemma from \cite[Proposition 4.1]{HaRa15}.
 
 \begin{lemma}\label{lemma:armijo}
Let $\beta\in (0,1)$, $\sigma\in (0,\frac{1}{2})$. 
Let $\mathbf u^{(j)}\in\mathbb R^n$ with $\Theta(\mathbf u^{(j)})>0$ and let $\mathbf d^{(j)}=\mathbf d(\mathbf u^{(j)})$ be computed by \eqref{eq:d}. Then, there exists a finite index $l\in\mathbb N$ with 
\begin{equation}\label{eq:Armijoinequality}
\Theta(\mathbf u^{(j)}+\beta^{l}\mathbf d^{(j)})\le (1-2\sigma \beta^l)\Theta(\mathbf u^{(j)}). 
\end{equation} 
 \end{lemma}
 \begin{proof}
According to Lemma \ref{lemma:descent}, it holds $\Theta'(\mathbf u^{(j)},\mathbf d^{(j)})\le-2\Theta(\mathbf u^{(j)})<0.$
The remainder of the proof follows \cite[Proof of Proposition 4.1]{HaRa15}.
\end{proof}

\begin{algorithm}[t]\floatname{algorithm}{Algorithm}
\vspace{0.2cm}
\footnotesize 
\begin{algorithmic}
\STATE{Choose a starting vector $\mathbf u^{(0)}\in\mathbb R^n$, parameters  $\beta\in (0,1)$, $\sigma\in (0,\frac{1}{2})$ and a tolerance $tol>0$ and set $j:=0$. In case of \texttt{hybridBSSN}, additionally choose $j_{max}\in\mathbb N$ and $t_{min}>0$.}
\IF{ \texttt{BSSN} is used \OR \texttt{hybridSSN} is used}
\STATE{Replace the modified index sets \eqref{eq:A+modified}--\eqref{eq:I-modified} by the index sets \eqref{eq:A+}--\eqref{eq:I-} in \eqref{eq:LCP}--\eqref{eq:BSSNz} and \eqref{eq:d}.}
\ENDIF
\WHILE{$\|\mathbf F(\mathbf u^{(j)})\|_2\ge tol $}
\STATE{Compute the Newton direction $\mathbf d^{(j)}$ from \eqref{eq:d}.}
\STATE{$t_j:=1$} 
\WHILE{$\Theta(\mathbf u^{(j)}+t_j\mathbf d^{(j)})>(1-2\sigma t_j)\Theta(\mathbf u^{(j)})$} \STATE{$t_j:=t_j\beta$} \ENDWHILE
\STATE{$\mathbf u^{(j+1)}:=\mathbf u^{(j)}+t_j \mathbf d^{(j)}$}
\STATE{$j:=j+1$}
\IF{\texttt{hybridSSN} is used \AND $j>j_{max}$ \AND $t_j<t_{min}$}
\STATE{Use the modified index sets \eqref{eq:A+modified}--\eqref{eq:I-modified} in \eqref{eq:LCP}--\eqref{eq:BSSNz} and \eqref{eq:d} for all following iterations.}
\ENDIF
\ENDWHILE 
\end{algorithmic}
\vspace{0.2cm}
\caption{The B-semismooth Newton methods \texttt{BSSN}, \texttt{modBSSN} and \texttt{hybridBSSN}}\label{algo1}
\end{algorithm}

The algorithm of the modified B-semismooth Newton method, in the following denoted by \texttt{modBSSN}, is stated in Algorithm \ref{algo1}. The feasibility of Algorithm \texttt{modBSSN} is guaranteed because of the lemmata stated above. 

\begin{remark} Pang \cite{Pa91} introduced a modified B-Newton method for a nonlinear complementarity problem. Han, Pang and Rangaraj \cite{HaPaRa92} interpreted this iteration as a generalized Newton method
\begin{equation}\nonumber
\mathbf F(\mathbf u^{(j)})+\tilde{\mathbf G}(\mathbf u^{(j)},\mathbf d^{(j)})=\mathbf 0,\qquad \mathbf u^{(j+1)}=\mathbf u^{(j)}+t_j\mathbf d^{(j)},\qquad j=0,1,\ldots,
\end{equation}
where $\tilde{\mathbf G}\colon\mathbb R^n\times \mathbb R^n\to\mathbb R^n$ fulfills the assumption that $\tilde{\mathbf G}(\mathbf u,\cdot)$ is surjective for each fixed $\mathbf u\in\mathbb R^n$, and
\begin{equation}\nonumber
2\big\langle \mathbf F(\mathbf u),\tilde{\mathbf G}(\mathbf u,\mathbf d)\big\rangle\ge \Theta'(\mathbf u,\mathbf d)
\end{equation}
for all $\mathbf u$, $\mathbf d\in\mathbb R^n$, see \cite[Section 2.3]{HaPaRa92}. In the very same way, our Algorithm \texttt{modBSSN} can be interpreted as a generalized Newton method with $\tilde{\mathbf G}(\mathbf u^{(j)},\mathbf d^{(j)})=\mathbf G(\mathbf u^{(j)})\mathbf d^{(j)}$ and $\mathbf G$ from \eqref{eq:G}, cf.\ Lemma \ref{lemma:descent}.
\end{remark}

\subsection{A B-semismooth Newton method and its feasibility}\label{sec:22}

The generalized formulation of the B-semismooth Newton method  \eqref{eq:SSNupdate}, \eqref{eq:BNewtonequation} from \cite{HaRa15} for the setting \eqref{eq:min}, in the following denoted by \texttt{BSSN}, is identical to Algorithm \texttt{modBSSN} replacing the modified index sets \eqref{eq:A+modified}--\eqref{eq:I-modified} by the original index sets \eqref{eq:A+}--\eqref{eq:I-} in \eqref{eq:LCP}--\eqref{eq:BSSNz} and \eqref{eq:d}, cf.\ Algorithm \ref{algo1} and \cite{HaRa15}. Analogously to the proofs in Section \ref{sec:21}, the Newton directions $\mathbf d^{(j)}$ can be shown to be uniquely determined and the Armijo stepsizes are well-defined because the Newton directions are descent directions w.r.t.\ the merit functional $\Theta$. Thus, the feasibility of the Algorithm \texttt{BSSN} is guaranteed.

\begin{remark}\label{remark:indexsets}
The modification of the index sets in Algorithm \texttt{modBSSN} is needed to prove global convergence without any additional requirements, see Section \ref{sec:globalconvergence}. Let $\mathbf u^*$ be the unique zero of $\mathbf F$ and let $\mathcal I^+(\mathbf u^*)\cup\mathcal I^-(\mathbf u^*)=\emptyset$, i.e.\ $\mathbf F$ is smooth at $\mathbf u^*$. Then, there exists a neighborhood $U$ of $\mathbf u^*$ where the index subsets \eqref{eq:A++}--\eqref{eq:I0-} are empty for all $\mathbf u\in U$, i.e.\ the modified index sets \eqref{eq:A+modified}--\eqref{eq:I-modified} match the original index sets \eqref{eq:A+}--\eqref{eq:I-}. Therefore, Algorithm \texttt{modBSSN} locally coincides with $\mathtt{BSSN}$ in a neighborhood of the zero $\mathbf u^*$ of $\mathbf F$ if $\mathcal I^+(\mathbf u^*)\cup\mathcal I^-(\mathbf u^*)=\emptyset$ and hence is a semismooth Newton method there. 
\end{remark}

\subsection{A globally convergent hybrid method}\label{sec:23}
The B-semismooth Newton method (Algorithm \texttt{BSSN}) from Section \ref{sec:22} is efficient in practice because the index sets $\mathcal I^\pm(\mathbf u^{(j)})$ in step $j$ are usually empty so that the generalized Newton equation simplifies to a system of linear equations of the size $|\mathcal A(\mathbf u^{(j)})|$. The size of the system of linear equations usually decreases in the course of the iteration. Nevertheless, the method may fail to converge, see Remark \ref{remark:indexsets} and Theorem \ref{th:convergenceoldSSN}. However, the global convergence of Algorithm \texttt{modBSSN} from Section \ref{sec:21} is ensured by Theorem \ref{th:convergence} but here a mixed linear complementarity problem has to be solved in each iteration, see \eqref{eq:d}. Additionally, in order to set up the matrix $\mathbf N$ and the vector $\mathbf z$ from \eqref{eq:BSSNN} and \eqref{eq:BSSNz}, $|\overline{\mathcal I}(\mathbf u^{(j)})|+1$ systems of linear equations of the size $|\overline{\mathcal A}(\mathbf u^{(j)}|$ with the same matrix have to be solved if $\overline{\mathcal I^\pm}(\mathbf u^{(j)})\neq \emptyset$. Note that in \eqref{eq:dA} resp.\ \eqref{eq:d} no additional system of linear equations has to be solved for the computation of $\mathbf d^{(j)}_{\overline{\mathcal A}}$. Nevertheless, Algorithm \texttt{modBSSN} is usually less efficient than Algorithm \texttt{BSSN}.

We suggest a hybrid method by starting with Algorithm \texttt{BSSN} and switching to Algorithm \texttt{modBSSN} when Algorithm \texttt{BSSN} begins to stagnate, by replacing the modified index sets \eqref{eq:A+modified}--\eqref{eq:I-modified} by the index sets \eqref{eq:A+}--\eqref{eq:I-} in \eqref{eq:LCP}--\eqref{eq:BSSNz} and \eqref{eq:d}. In our numerical experiments, we switch to Algorithm \texttt{modBSSN} if the number of Newton steps exceeds a limit $j_{max}\in\mathbb N$ and if the chosen stepsize is smaller than a threshold $t_{min}>0$, i.e.\ if $j> j_{max}$ and $t_j< t_{min}$. In the sequel, this hybrid method is called \texttt{hybridBSSN}. An overview of the proposed methods is given in Algorithm \ref{algo1}. Similar hybrid methods, combining the fast local convergence properties of a local semismooth Newton method with the globally convergent generalized Newton method from \cite{HaPaRa92} were proposed by Qi \cite{Qi93} and Ito and Kunisch \cite{ItKu09}.

\section{Global convergence and local convergence speed}\label{sec:globalconvergence}

In this section, we consider the convergence properties of the algorithms from Section \ref{sec:feasibility}. 
\subsection{Convergence of the modified B-semismooth Newton method}\label{sec:31}
In the following, we address the global convergence of Algorithm \texttt{modBSSN} and its convergence speed in a neighborhood of the zero of $\mathbf F$.
Concerning the boundedness of the sequence of Newton directions $\{\mathbf d^{(j)}\}_j$, we cite \cite[Proposition 4.6]{HaRa15}.

\begin{lemma}\label{lemma:dbounded}
Let $\mathbf u\in\mathbb R^n$ and $\mathbf d=\mathbf d(\mathbf u)$ be the solution to \eqref{eq:d}. Then, there exists a constant $C=C(n)>0$ independent of $\mathbf u$, with
\begin{equation}
\|\mathbf d\|_2\le C\|\mathbf F(\mathbf u)\|_2.
\end{equation}
\end{lemma}
\begin{proof}
The proof follows \cite[Proof of Proposition 4.6]{HaRa15} by substituting the index sets $\mathcal A^\pm$, $\mathcal I^\circ$ and $\mathcal I^\pm$ by the modified index sets $\overline{\mathcal A^\pm}$, $\overline{\mathcal I^\circ}$ and $\overline{\mathcal I^\pm}$ respectively. An inspection of the proof of Lemma 3.3 from \cite{HaRa15} and Assumption \ref{ass:1} shows that the Rayleigh quotient of $\mathbf N=\mathbf N(\mathbf u)$ from \eqref{eq:BSSNN} is bounded from below.
\end{proof}

In the following theorem, we present our main result on the global convergence of Algorithm \texttt{modBSSN}.

\begin{theorem}\label{th:convergence}
Let $\mathbf u^*\in\mathbb R^n$ be an accumulation point of the sequence of iterates $\{\mathbf u^{(j)}\}_j$ produced by Algorithm $\mathtt{modBSSN}$. Then, we have $\Theta(\mathbf u^*)=0$.
\end{theorem}
\begin{proof}
We proceed analogously to the proof of \cite[Theorem 1]{Pa91} and we also use the proof of \cite[Proposition 1]{Pa91}. We suppose $\Theta(\mathbf u^{(j)})>0$ for all $j$, because otherwise the claim is proven. Because of the Armijo rule \eqref{eq:Armijoinequality}, the sequence $\{\Theta(\mathbf u^{(j)})\}_j$ strictly decreases and is bounded from below by $0$, i.e.\ convergent. Let $t_j=\beta^{l_j}$ be the computed Armijo stepsize in step $j$. From the Armijo rule \eqref{eq:Armijoinequality}, it follows
\begin{equation}\nonumber
0<2\sigma t_j \Theta(\mathbf u^{(j)})\le \Theta(\mathbf u^{(j)})-\Theta(\mathbf u^{(j+1)})\to 0,\quad j\to\infty.
\end{equation}
Therefore, we have
\begin{equation}\nonumber
\lim\limits_{j\to\infty} t_j\Theta(\mathbf u^{(j)})=0.
\end{equation}
The level set $L_\Theta(\mathbf u^{(0)})=\{\mathbf u\in\mathbb R^n: \Theta(\mathbf u)\le\Theta(\mathbf u^{(0)})\}$ is bounded by Assumption \ref{ass:1}, implying that the sequence $\{\mathbf u^{(j)}\}_j$ is bounded and has an accumulation point $\mathbf u^*$. Let $\{\mathbf u^{(j)}\}_{j\in J}$ be a subsequence converging to $\mathbf u^*$. If the stepsizes $t_j$ are bounded away from zero, i.e.\ we have $\limsup_{j\to\infty, j\in J} t_j>0$, it directly follows $\Theta(\mathbf u^*)=0$.

Let us now consider the case $\limsup_{j\to\infty, j\in J} t_j=0$. Without loss of generality, we suppose $\lim_{j\to\infty, j\in J} t_j=0$. By the Armijo rule \eqref{eq:Armijoinequality}, we have for all $j\in J$
\begin{equation}\label{eq:convergenceproof1}
\Theta(\mathbf u^{(j)})-\Theta(\mathbf u^{(j)}+\beta^{l_j-1}\mathbf d^{(j)})<2\sigma\beta^{l_j-1}\Theta(\mathbf u^{(j)}).
\end{equation}
We define $\hat{\mathbf u}^{(j)}:=\mathbf u^{(j)}+\beta^{l_j-1}\mathbf d^{(j)}$. The sequence $\{\mathbf d^{(j)}\}_j$ of Newton directions is bounded because of Lemma \ref{lemma:dbounded}, implying that $\mathbf u^*$ is the limit of the subsequence $\{\hat{\mathbf u}^{(j)}\}_{j\in J}$. Therefore, without loss of generality we have
\begin{align*}
\mathcal A^+_+(\mathbf u^*)&\subset\mathcal A^+_+(\mathbf u^{(j)})\cap\mathcal A_+^+(\hat{\mathbf u}^{(j)}),\\
\mathcal A^-_-(\mathbf u^*)&\subset\mathcal A^-_-(\mathbf u^{(j)})\cap\mathcal A_-^-(\hat{\mathbf u}^{(j)}),\\
\mathcal I^\circ_+(\mathbf u^*)&\subset\mathcal I^\circ_+(\mathbf u^{(j)})\cap\mathcal I_+^\circ(\hat{\mathbf u}^{(j)}),\\
\mathcal I^\circ_-(\mathbf u^*)&\subset\mathcal I^\circ_-(\mathbf u^{(j)})\cap\mathcal I_-^\circ(\hat{\mathbf u}^{(j)}),
\end{align*}
 for all $j\in J$ large enough.
Now we consider
\begin{equation}\label{eq:convergenceproof2}
\Theta(\mathbf u^{(j)})-\Theta(\hat{\mathbf u}^{(j)})=\sum\limits_{i=1}^8\tilde T_i,
\end{equation}
where
\begin{align*}
\begin{split}
\tilde T_1&:=\sum\limits_{k\in\overline{\mathcal A}(\mathbf u^*)} \big(F_k(\mathbf u^{(j)})^2-F_k(\hat{\mathbf u}^{(j)})^2\big),\\
\tilde T_3&:=\sum\limits_{k\in\mathcal I^+(\mathbf u^*)} \hspace{-0.1cm}\big(F_k(\mathbf u^{(j)})^2-F_k(\hat{\mathbf u}^{(j)})^2\big),\\
\tilde T_5&:=\sum\limits_{k\in\mathcal A^+_+(\mathbf u^*)} \hspace{-0.1cm}\big(F_k(\mathbf u^{(j)})^2-F_k(\hat{\mathbf u}^{(j)})^2\big),\\
\tilde T_7&:=\sum\limits_{k\in\mathcal I^\circ_+(\mathbf u^*)} \hspace{-0.1cm}\big(F_k(\mathbf u^{(j)})^2-F_k(\hat{\mathbf u}^{(j)})^2\big),\\
\end{split}
\hspace{-0.3cm}
\begin{split}
&\tilde T_2:=\sum\limits_{k\in\overline{\mathcal I^\circ}(\mathbf u^*)}\hspace{-0.1cm} \big(F_k(\mathbf u^{(j)})^2-F_k(\hat{\mathbf u}^{(j)})^2\big),\\
&\tilde T_4:=\sum\limits_{k\in\mathcal I^-(\mathbf u^*)}\hspace{-0.1cm} \big(F_k(\mathbf u^{(j)})^2-F_k(\hat{\mathbf u}^{(j)})^2\big),\\
&\tilde T_6:=\sum\limits_{k\in\mathcal A^-_-(\mathbf u^*)} \hspace{-0.1cm}\big(F_k(\mathbf u^{(j)})^2-F_k(\hat{\mathbf u}^{(j)})^2\big),\\
&\tilde T_8:=\sum\limits_{k\in\mathcal I^\circ_-(\mathbf u^*)} \hspace{-0.1cm}\big(F_k(\mathbf u^{(j)})^2-F_k(\hat{\mathbf u}^{(j)})^2\big).
\end{split}
\end{align*}
In the following, we estimate each sum from below. Finally, we prove the claim by using \eqref{eq:convergenceproof1} and by taking the limit $j\to\infty$, $j\in J$.

If $k\in\overline{\mathcal A}(\mathbf u^*)$, we have for $j\in J$ large enough $k\in\overline{\mathcal A}(\mathbf u^{(j)})$, $k\in\mathcal A^+_+(\mathbf u^{(j)})$ or $k\in\mathcal A^-_-(\mathbf u^{(j)})$. Using \eqref{eq:help1}, \eqref{eq:help2}, \eqref{eq:help4} and \eqref{eq:help5}, we obtain
\begin{align*}
\tilde T_1&=\sum\limits_{k\in\overline{\mathcal A}(\mathbf u^*)} \big((\gamma(\nabla g(\mathbf u^{(j)}))_k\pm \gamma w_k)^2-(\gamma(\nabla g(\hat{\mathbf u}^{(j)}))_k\pm \gamma w_k)^2\big)\\
&=\sum\limits_{k\in\overline{\mathcal A}(\mathbf u^*)}-2\beta^ {l_j-1}(\gamma(\nabla g(\mathbf u^{(j)}))_k\pm \gamma w_k)\gamma(\nabla^2 g(\mathbf u^{(j)})\mathbf d^{(j)})_k\\
& \phantom{=}+ o(\|\hat{\mathbf u}^{(j)}-\mathbf u^{(j)}\|_2)\\
&\ge 2\beta^{l_j-1}\sum\limits_{k\in\overline{\mathcal A}(\mathbf u^*)}F_k(\mathbf u^{(j)})^2+o(\beta^{l_j-1}\|\mathbf d^{(j)}\|_2),\quad j\to\infty, j\in J.
\end{align*}
Analogously it follows with \eqref{eq:help4} and \eqref{eq:help5}
\begin{equation}\nonumber
\tilde T_5\ge 2\beta^{l_j-1} \sum\limits_{k\in\mathcal A^+_+(\mathbf u^*)}F_k(\mathbf u^{(j)})^2+o(\beta^{l_j-1}\|\mathbf d^{(j)}\|_2),\quad j\to\infty, j\in J,
\end{equation}
and
\begin{equation}\nonumber
\tilde T_6\ge 2\beta^{l_j-1} \sum\limits_{k\in\mathcal A^-_-(\mathbf u^*)}F_k(\mathbf u^{(j)})^2+o(\beta^{l_j-1}\|\mathbf d^{(j)}\|_2),\quad j\to\infty, j\in J.
\end{equation}
For $k\in \overline{I^\circ}(\mathbf u^*)$, we have to consider the cases $k\in\overline{\mathcal I^\circ}(\mathbf u^{(j)})$, $k\in\mathcal I^\circ_+(\mathbf u^{(j)})$ and $k\in \mathcal I^\circ_-(\mathbf u^{(j)})$. With \eqref{eq:help3} and \eqref{eq:help6}, we have
\begin{align*}
\tilde T_2&=\sum\limits_{k\in\overline{\mathcal I^\circ}(\mathbf u^*)}\big((u_k^{(j)})^2-(\hat{u}_k^{(j)})^2\big)=\sum\limits_{k\in\overline{\mathcal I^\circ}(\mathbf u^*)}\Big(-2\beta^{l_j-1}u_k^{(j)}d_k^{(j)}-(\beta^{l_j-1}d_k^{(j)})^2\Big)\\
&\ge 2\beta^{l_j-1} \sum\limits_{k\in\overline{\mathcal I^\circ}(\mathbf u^*)}(u_k^{(j)})^2-\sum\limits_{k\in\overline{\mathcal I^\circ}(\mathbf u^*)}(\beta^{l_j-1}d_k^{(j)})^2\\
&= 2\beta^{l_j-1} \sum\limits_{k\in\overline{\mathcal I^\circ}(\mathbf u^*)}F_k(\mathbf u^{(j)})^2-\sum\limits_{k\in\overline{\mathcal I^\circ}(\mathbf u^*)}(\beta^{l_j-1}d_k^{(j)})^2.
\end{align*}
Accordingly, it follows with \eqref{eq:help6}
\begin{equation}\nonumber
\tilde T_7\ge 2\beta^{l_j-1} \sum\limits_{k\in\mathcal I^\circ_+(\mathbf u^*)}F_k(\mathbf u^{(j)})^2-\sum\limits_{k\in\mathcal I^\circ_+(\mathbf u^*)}(\beta^{l_j-1}d_k^{(j)})^2
\end{equation}
and
\begin{equation}\nonumber
\tilde T_8\ge 2\beta^{l_j-1} \sum\limits_{k\in\mathcal I^\circ_-(\mathbf u^*)}F_k(\mathbf u^{(j)})^2-\sum\limits_{k\in\mathcal I^\circ_-(\mathbf u^*)}(\beta^{l_j-1}d_k^{(j)})^2.
\end{equation}
In the following, we treat the sum $\tilde T_3$. For $k\in\mathcal I^+(\mathbf u^*)$, we may assume without loss of generality
\begin{align*}
k\in\Big(\mathcal I^+(\mathbf u^{(j)})\cup\mathcal I^\circ(\mathbf u^{(j)})\cup\mathcal A^+(\mathbf u^{(j)})\Big)\cap\Big(\mathcal I^+(\hat{\mathbf u}^{(j)})\cup\mathcal I^\circ(\hat{\mathbf u}^{(j)})\cup\mathcal A^+(\hat{\mathbf u}^{(j)})\Big).
\end{align*}
We split $\mathcal I^+(\mathbf u^*)=S_1(\mathbf u^*)\cap S_2(\mathbf u^*)\cap S_3(\mathbf u^*)$, where
\begin{align*}
S_1(\mathbf u^*)&:=\{k: u_k^*=\gamma(\nabla g(\mathbf u^*))_k+\gamma w_k>0\},\\
S_2(\mathbf u^*)&:=\{k: u_k^*=\gamma(\nabla g(\mathbf u^*))_k+\gamma w_k=0\},\\
S_3(\mathbf u^*)&:=\{k: u_k^*=\gamma(\nabla g(\mathbf u^*))_k+\gamma w_k<0\}.
\end{align*}
For $k\in S_1(\mathbf u^*)$, we may assume with \eqref{eq:Fmin}
\begin{align*}
F_k(\mathbf u^{(j)})&=\min\{u_k^{(j)},\gamma (\nabla g(\mathbf u^{(j)}))_k+\gamma w_k\}>0,\\
F_k(\hat{\mathbf u}^{(j)})&=\min\{\hat{u}_k^{(j)},\gamma (\nabla g(\hat{\mathbf u}^{(j)}))_k+\gamma w_k\}>0.
\end{align*}
Therefore, we have
\begin{align*}\nonumber
F_k(\mathbf u^{(j)})^2-F_k(\hat{\mathbf u}^{(j)})^2=& \min\{(u_k^{(j)})^2,(\gamma (\nabla g(\mathbf u^{(j)}))_k+\gamma w_k)^2\}\\
&-\min\{(\hat{u}_k^{(j)})^2,(\gamma (\nabla g(\hat{\mathbf u}^{(j)}))_k+\gamma w_k)^2\}.
\end{align*}
In the case $k\in\mathcal I^\circ(\mathbf u^{(j)})$, we have $k\in\overline{\mathcal I^\circ}(\mathbf u^{(j)})$ or $k\in\mathcal I_-^\circ(\mathbf u^{(j)})$ because $F_k(\mathbf u^{(j)})>0$. With \eqref{eq:help3} and \eqref{eq:help6}, we have
\begin{align*}
F_k(\mathbf u^{(j)})^2-F_k(\hat{\mathbf u}^{(j)})^2&\ge (u_k^{(j)})^2-(\hat{u}_k^{(j)})^2=-2\beta^{l_j-1} u_k^{(j)} d_k^{(j)}-(\beta^{l_j-1}d_k^{(j)})^2\\
&\ge 2\beta^{l_j-1} (u_k^{(j)})^2-(\beta^{l_j-1}d_k^{(j)})^2.
\end{align*}
For $k\in\mathcal A^+(\mathbf u^{(j)})$, it follows $k\in\overline{\mathcal A^+}(\mathbf u^{(j)})$ because $F_k(\mathbf u^{(j)})>0$. Hence, one has with \eqref{eq:help1}
\begin{align*}
&F_k(\mathbf u^{(j)})^2-F_k(\hat{\mathbf u}^{(j)})^2\\
\ge &\ (\gamma(\nabla g(\mathbf u^{(j)}))_k+\gamma w_k)^2-(\gamma(\nabla g(\hat{\mathbf u}^{(j)}))_k+\gamma w_k)^2\\
= & -2\beta^{l_j-1} (\gamma(\nabla g(\mathbf u^{(j)}))_k+\gamma w_k) \gamma(\nabla ^2 g(\mathbf u^{(j)})\mathbf d^{(j)})_k+ o(\|\beta^{l_j-1}\mathbf d^{(j)}\|_2)\\
= &\ 2\beta^{l_j-1} F_k(\mathbf u^{(j)})^2+o(\beta^{l_j-1}\|\mathbf d^{(j)}\|_2),\quad j\to\infty, j\in J.
\end{align*}
If $k\in\mathcal I^+(\mathbf u^{(j)})$, we have $F_k(\mathbf u^{(j)})=u_k^{(j)}=\gamma(\nabla g(\mathbf u^{(j)}))_k+\gamma w_k$ and we have either $d_k^{(j)}=-u_k^{(j)}$ or $\gamma(\nabla^2 g(\mathbf u^{(j)})\mathbf d^{(j)})_k=-(\gamma(\nabla g(\mathbf u^{(j)}))_k+\gamma w_k)$, see \eqref{eq:LCP} and \eqref{eq:dx}. As in the cases $k\in\mathcal I^\circ (\mathbf u^{(j)})$ and $k\in\mathcal A^+(\mathbf u^{(j)})$, we conclude
\begin{align*}
F_k(\mathbf u^{(j)})^2-F_k(\hat{\mathbf u}^{(j)})^2\ge 2\beta^{l_j-1} F_k(\mathbf u^{(j)})^2+o(\beta^{l_j-1}\|\mathbf d^{(k)}\|_2),\quad j\to\infty, j\in J. 
\end{align*}
Altogether, we get
\begin{align*}
&\sum\limits_{k\in S_1(\mathbf u^*)} \big(F_k(\mathbf u^{(j)})^2-F_k(\hat{\mathbf u}^{(j)})^2\big)\\
\ge& \ 2 \beta^{l_j-1} \sum\limits_{k\in S_1(\mathbf u^*)} F_k(\mathbf u^{(j)})^2+o(\beta^{l_j-1}\|\mathbf d^{(j)}\|_2),\quad j\to\infty, j\in J.
\end{align*}
For $k\in S_2(\mathbf u^*)$, we have with Lipschitz-constant $L$ of $F_k$
\begin{align*}
F_k(\mathbf u^{(j)})^2-F_k(\hat{\mathbf u}^{(j)})^2&=\big(F_k(\mathbf u^{(j)})-F_k(\hat{\mathbf u}^{(j)})\big)\big(F_k(\mathbf u^{(j)})+F_k(\hat{\mathbf u}^{(j)})\big)\\
&\le L\|\hat{\mathbf u}^{(j)}-\mathbf u^{(j)}\|_2 \big|F_k(\mathbf u^{(j)})+F_k(\hat{\mathbf u}^{(j)})\big|\\
&=L\beta^{l_j-1}\|\mathbf d^{(j)}\|_2 \big|F_k(\mathbf u^{(j)})+F_k(\hat{\mathbf u}^{(j)})\big|.
\end{align*}
It follows
\begin{align*}
\lim\limits_{j\to\infty, j\in J} \sum\limits_{k\in S_2(\mathbf u^*)} \frac{F_k(\mathbf u^{(j)})^2-F_k(\hat{\mathbf u}^{(j)})^2}{\beta^{l_j-1}}=0.
\end{align*}
Let now $k\in S_3(\mathbf u^*)$. We may assume $u_k^{(j)}<0$, $\hat{u}_k^{(j)}<0$, $\gamma(\nabla g(\mathbf u^{(j)}))_k+\gamma w_k<0$ and $\gamma(\nabla g(\hat{\mathbf u}^{(j)}))_k+\gamma w_k<0$. With \eqref{eq:Fmin}, one has
\begin{align*}
F_k(\mathbf u^{(j)})^2-F_k(\hat{\mathbf u}^{(j)})^2=& \max\{(u_k^{(j)})^2, (\gamma(\nabla g(\mathbf u^{(j)}))_k+\gamma w_k)^2\}\\
&-\max\{(\hat{u}_k^{(j)})^2, (\gamma(\nabla g(\hat{\mathbf u}^{(j)}))_k+\gamma w_k)^2\}.
\end{align*}
First, we treat the case $(\hat{u}_k^{(j)})^2<(\gamma(\nabla g(\hat{\mathbf u}^{(j)}))_k+\gamma w_k)^2$. We have to consider the cases $k\in\mathcal I^\circ_+(\mathbf u^{(j)})$, $k\in\mathcal A_+^+(\mathbf u^{(j)})$ and $k\in\{k\in\mathcal I^+(\mathbf u^{(j)}): F_k(\mathbf u^{(j)})<0)\}$. With \eqref{eq:help4}, we have
\begin{align*}
&F_k(\mathbf u^{(j)})^2-F_k(\hat{\mathbf u}^{(j)})^2\\
\ge&\ (\gamma(\nabla g(\mathbf u^{(j)}))_k+\gamma w_k)^2-(\gamma(\nabla g(\hat{\mathbf u}^{(j)}))_k+\gamma w_k)^2\\
=&\ -2\beta^{l_j-1}(\gamma(\nabla g(\mathbf u^{(j)}))_k+\gamma w_k)\gamma(\nabla^2 g(\mathbf u^{(j)})\mathbf d^{(j)})_k+o(\beta^{l_j-1}\|\mathbf d^{(j)}\|_2)\\
\ge &\ 2\beta^{l_j-1}(\gamma(\nabla g(\mathbf u^{(j)}))_k+\gamma w_k)^2 +o(\beta^{l_j-1}\|\mathbf d^{(j)}\|_2),\quad j\to\infty, j\in J.
\end{align*}
Second, we consider the case $(\hat{u}_k^{(j)})^2\ge (\gamma(\nabla g(\hat{\mathbf u}^{(j)}))_k+\gamma w_k)^2$. With \eqref{eq:help6}, we have analogously
\begin{align*}
F_k(\mathbf u^{(j)})^2-F_k(\hat{\mathbf u}^{(j)})^2&\ge (u_k^{(j)})^2-(\hat{u}_k^{(j)})^2=-2\beta^{l_j-1}u_k^{(j)}d_k^{(j)}-(\beta^{l_j-1}d_k^{(j)})^2\\
&\ge 2\beta^{l_j-1}(u_k^{(j)})^2-(\beta^{l_j-1}d_k^{(j)})^2.
\end{align*}
Altogether, we obtain
\begin{align*}
&\sum\limits_{k\in S_3(\mathbf u^*)} \big(F_k(\mathbf u^{(j)})^2-F_k(\hat{\mathbf u}^{(j)})^2\big)\\
\ge &\ 2\beta^{l_j-1} \sum\limits_{k\in S_3(\mathbf u^*)} \min\{(u_k^{(j)})^2,(\gamma(\nabla g(\mathbf u^{(j)}))_k+\gamma w_k)^2\}\\
& +o(\beta^{l_j-1}\|\mathbf d^{(j)}\|_2), j\to\infty, j\in J.
\end{align*}
By symmetry, we can treat the sum $\tilde T_4$ similarly. For $j\to\infty, j\in J$, we get
\begin{align*}
&\sum\limits_{\{k\in \mathcal I^-(\mathbf u^*): u_k^*\neq 0\}}\big(F_k(\mathbf u^{(j)})^2-F_k(\hat{\mathbf u}^{(j)})^2\big) \\
\ge &\  2\beta^{l_j-1}\sum\limits_{\{k\in \mathcal I^-(\mathbf u^*): u_k^*\neq 0\}}\min\{(u_k^{(j)})^2,(\gamma(\nabla g(\mathbf u^{(j)}))_k-\gamma w_k)^2\}\\
&+o(\beta^{l_j-1}\|\mathbf d^{(j)}\|_2),
\end{align*}
and 
\begin{align*}
\lim\limits_{j\to\infty, j\in J} \sum\limits_{\{k\in \mathcal I^+(\mathbf u^*): u_k^*=0\}} \frac{F_k(\mathbf u^{(j)})^2-F_k(\hat{\mathbf u}^{(j)})^2}{\beta^{l_j-1}}=0.
\end{align*}

Finally, we divide both sides of the inequality \eqref{eq:convergenceproof1} by $\beta^{l_j-1}$ and take the limit $j\to\infty$, $j\in J$, obtaining with \eqref{eq:convergenceproof2} and the previous estimates
\begin{align*}
2\Theta(\mathbf u^*)\le 2\sigma \Theta(\mathbf u^*).
\end{align*}
Here, we use the fact that the sequence $\{\mathbf d^{(j)}\}_j$ is bounded, implying
\begin{align*}
\lim\limits_{j\to\infty, j\in J}\frac{o(\beta^{l_j-1}\|\mathbf d^{(j)}\|_2)}{\beta^{l_j-1}}=\lim\limits_{j\to\infty, j\in J}\frac{o(\beta^{l_j-1}\|\mathbf d^{(j)}\|_2)}{\beta^{l_j-1}\|\mathbf d^{(j)}\|_2}\|\mathbf d^{(j)}\|_2=0.
\end{align*}
The choice $\sigma<1/2$ implies $\Theta(\mathbf u^*)=0$, finishing the proof.
\end{proof}

As a consequence of the last theorem, we can argue that the stepsizes in Algorithm \texttt{modBSSN} are eventually chosen equal to $1$. In the following theorem, we additionally assume that $g$ is more regular and that $\mathbf F$ is smooth at the unique zero $\mathbf u^*$, i.e.\ $\mathcal I^+(\mathbf u^*)\cup\mathcal I^-(\mathbf u^*)=\emptyset$. 

\begin{theorem}\label{th:tequalto1}
Let $g$ be three times continuously differentiable. Let $\{\mathbf u^{(j)}\}_j$ be a sequence produced by Algorithm $\mathtt{modBSSN}$ converging to a limit point $\mathbf u^*$ with $\mathcal I^+(\mathbf u^*)\cup\mathcal I^-(\mathbf u^*)=\emptyset$. Then, there exists an index $j_0\in\mathbb N$ such that $t_j=1$ for all $j\ge j_0$.
\end{theorem}
\begin{proof}
We proceed as in the proof of \cite[Theorem 2]{Pa91}. Inspired by loc.\ cit., we show that for all $j$ large enough, we have
\begin{equation}
\Theta(\mathbf u^{(j)})-\Theta(\mathbf u^{(j)}+\mathbf d^{(j)})\ge 2\sigma \Theta(\mathbf u^{(j)}).
\end{equation}
We show the claim by contradiction. Let the subsequence $\{\mathbf u^{(j)}\}_{j\in J}$ fulfill
\begin{equation}\label{eq:Armijot1}
\Theta(\mathbf u^{(j)})-\Theta(\mathbf u^{(j)}+\mathbf d^{(j)})< 2\sigma \Theta(\mathbf u^{(j)})
\end{equation}
for all $j\in J$ large enough. Because of Lemma \ref{lemma:dbounded}, we have $\|\mathbf d^{(j)}\|_2\le C\|\mathbf F(\mathbf u^{(j)})\|_2$ with a constant $C>0$. Therefore, with $\hat{\mathbf u}^{(j)}:=\mathbf u^{(j)}+\mathbf d^{(j)}$, the sequence $\{\hat{\mathbf u}^{(j)}\}_{j\in J}$ has the limit $\mathbf u^*$. We consider
\begin{equation}
\Theta(\mathbf u^{(j)})-\Theta(\hat{\mathbf u}^{(j)})=\sum\limits_{i=1}^2 \hat{T}_i,
\end{equation}
where 
\begin{align*}
\hat{T}_1&:=\sum\limits_{k\in\mathcal A(\mathbf u^*)} \big(F_k(\mathbf u^{(j)})^2-F_k(\hat{\mathbf u}^{(j)})^2\big),\\
\hat{T}_2&:=\sum\limits_{k\in\mathcal I^\circ(\mathbf u^*)}\big( F_k(\mathbf u^{(j)})^2-F_k(\hat{\mathbf u}^{(j)})^2\big).
\end{align*}
Because of Theorem \ref{th:convergence}, we have
\begin{align*}
\mathcal A^+(\mathbf u^*)&=\{k: 0=\gamma(\nabla g(\mathbf u^*))_k+\gamma w_k<u_k^*\},\\
\mathcal A^-(\mathbf u^*)&=\{k: 0=\gamma(\nabla g(\mathbf u^*))_k-\gamma w_k>u_k^*\},\\
\mathcal I^\circ(\mathbf u^*)&=\{k: \gamma(\nabla g(\mathbf u^*))_k-\gamma w_k<u_k^*=0<\gamma(\nabla g(\mathbf u^*))_k+\gamma w_k\}.
\end{align*}
For all $j\in J$ large enough, we have
\begin{align*}
\mathcal A^+(\mathbf u^*)&\subset\overline{\mathcal A^+}(\mathbf u^{(j)})\cap\overline{\mathcal A^+}(\hat{\mathbf u}^{(j)}),\\
\mathcal A^-(\mathbf u^*)&\subset\overline{\mathcal A^-}(\mathbf u^{(j)})\cap\overline{\mathcal A^-}(\hat{\mathbf u}^{(j)}),\\
\mathcal I^\circ(\mathbf u^*)&\subset\overline{\mathcal I^\circ}(\mathbf u^{(j)})\cap\overline{\mathcal I^\circ}(\hat{\mathbf u}^{(j)}).
\end{align*}
Lemma \ref{lemma:dbounded} implies the boundedness of the subsequence $\{\mathbf d^{(j)}/\|\mathbf F(\mathbf u^{(j)})\|_2\}_{j\in J}$  of quotients and without loss of generality, this subsequence has a limit $\tilde{\mathbf d}\in\mathbb R^n$ and the subsequence $\{\mathbf F(\mathbf u^{(j)})/\|\mathbf F(\mathbf u^{(j)})\|_2\}_{j\in J}$ of unit vectors tends to a unit vector $\tilde{\mathbf F}\in\mathbb R^n$.

Similar to the proof of Theorem \ref{th:convergence}, we estimate the sums $\hat T_1$ and $\hat T_2$. First, we treat the sum $\hat{T}_1$. Because $k\in\mathcal A(\mathbf u^*)\subset \overline{\mathcal A}(\mathbf u^{(j)})\cap\overline{\mathcal A}(\hat{\mathbf u}^{(j)})$, we have $F_k(\mathbf u^{(j)})+\gamma(\nabla^2 g(\mathbf u^{(j)})\mathbf d^{(j)})_k=0$. Dividing by $\|\mathbf F(\mathbf u^{(j)})\|_2$ and taking the limit $j\to\infty, j\in J$, it follows
\begin{align*}
\tilde{\mathbf F}_k+\gamma(\nabla^2 g(\mathbf u^*)\tilde{\mathbf d})_k= 0.
\end{align*}
There exists a vector $\mathbf v$ on the line segment between $\mathbf u^{(j)}$ and $\hat{\mathbf u}^{(j)}$ with
\begin{align*}
&\hat{T}_1=\sum\limits_{k\in\mathcal A(\mathbf u^*)}\big( F_k(\mathbf u^{(j)})^2-F_k(\hat{\mathbf u}^{(j)})^2\big)\\
=&\sum\limits_{k\in\mathcal A(\mathbf u^*)}\Big(-2F_k(\mathbf u^{(j)})(\gamma\nabla^2 g(\mathbf u^{(j)})\mathbf d^{(j)})_k-(\gamma\nabla^2 g(\mathbf v)\mathbf d^{(j)})_k^2\\
& \phantom{= \sum\sum} -F_k(\mathbf v)\gamma\sum\limits_{l,m=1}^n \frac{\partial^3 g(\mathbf v)}{\partial u_l\partial u_m\partial u_k} d_l^{(j)} d_m^{(j)}\Big)\\
=&\sum\limits_{k\in\mathcal A(\mathbf u^*)}\Big(2F_k(\mathbf u^{(j)})^2-(\gamma\nabla^2 g(\mathbf v)\mathbf d^{(j)})_k^2-F_k(\mathbf v)\gamma\sum\limits_{l,m=1}^n \frac{\partial^3 g(\mathbf v)}{\partial u_l\partial u_m\partial u_k} d_l^{(j)} d_m^{(j)}\Big).
\end{align*}
Dividing by $\|\mathbf F(\mathbf u^{(j)})\|_2^2$ and taking the limit $j\to\infty$, $j\in J$, it follows
\begin{align*}
\lim\limits_{j\to\infty, j\in J}\frac{\hat{T}_1}{\|\mathbf F(\mathbf u^{(j)})\|_2^2}=\sum\limits_{k\in\mathcal A(\mathbf u^*)} \tilde{F}_k^2.
\end{align*}
Now we consider the sum $\hat T_2$. We have $k\in\mathcal I^\circ(\mathbf u^*)\subset \overline{\mathcal I^\circ}(\mathbf u^{(j)})\cap\overline{\mathcal I^\circ}(\hat{\mathbf u}^{(j)})$ and
\begin{align*}
\hat{T}_2&=\sum\limits_{k\in\mathcal I^\circ(\mathbf u^*)}\big( F_k(\mathbf u^{(j)})^2-F_k(\hat{\mathbf u}^{(j)})^2\big)=\sum\limits_{k\in\mathcal I^\circ(\mathbf u^*)}\big( (u_k^{(j)})^2-(\hat{u}_k^{(j)})^2\big)\\
&=\sum\limits_{k\in\mathcal I^\circ(\mathbf u^*)}F_k(\mathbf u^{(j)})^2.
\end{align*}
Therefore, we have
\begin{align*}
\lim\limits_{j\to\infty, j\in J}\frac{\hat{T}_2}{\|\mathbf F(\mathbf u^{(j)})\|_2^2}=\sum\limits_{k\in\mathcal I^\circ(\mathbf u^*)} \tilde{F}_k^2.
\end{align*}

Finally, we divide both sides of the inequality \eqref{eq:Armijot1} by $\|\mathbf F(\mathbf u^{(j)})\|_2^2$ and take the limit $j\to\infty$, $j\in J$, obtaining
\begin{align*}
\|\tilde{\mathbf F}\|_2^2 \le 2\sigma \|\tilde{\mathbf F}\|_2^2
\end{align*}
which is a contradiction to $\|\tilde{\mathbf F}\|_2=1$ and the choice $\sigma<1/2$ in the Armijo rule \eqref{eq:Armijoinequality}, finishing the proof.
\end{proof}

Now we consider the locally quadratic convergence of Algorithm \texttt{modBSSN} in the case that the stepsizes $t_j$ are eventually chosen equal to $1$, i.e.\ according to Theorem \ref{th:tequalto1} especially in the case $\mathcal I^+(\mathbf u^*)\cup\mathcal I^-(\mathbf u^*)=\emptyset$. In the following theorem, we need the bounded invertibility of $\mathbf G(\mathbf u)$ from \eqref{eq:G} in a neighborhood of the zero $\mathbf u^*$ of $\mathbf F$. Because $\mathbf M:=\nabla^2 g(\mathbf u)$ is symmetric and positive definite, the inverse of $\mathbf G$ at $\mathbf u$ is bounded by a constant $\tilde C>0$
\begin{equation}\label{eq:estimateinverse}
\|\mathbf G(\mathbf u)^{-1}\|_2\le \|\mathbf M_{\mathcal B,\mathcal B}^{-1}\|_2\Big(\frac{1}{\gamma}+\|\mathbf M_{\mathcal B,\mathcal C}\|_2\Big)+1\le \|\mathbf M^{-1}\|_2\Big(\frac{1}{\gamma}+\|\mathbf M\|_2\Big)+1\le\tilde C,
\end{equation}
see \cite[Proposition 3.11]{GrLo08} and \cite[Lemma 3.6]{MuHaMaPi13}. The boundedness follows from Assumption \ref{ass:1}. For the following theorem, we need again the additional assumption that $g$ is three times continuously differentiable.

\begin{theorem}\label{th:quadraticconvergence}
Let $g$ be three times continuously differentiable and let the stepsizes $t_j$ be chosen equal to $1$ for all $j$ large enough. Let $\{\mathbf u^{(j)}\}_j$ be a sequence produced by Algorithm $\mathtt{modBSSN}$ converging to $\mathbf u^*$. Then, there exists a constant $C>0$ so that locally quadratic convergence is achieved, i.e.\ for all $j$ large enough, we have
\begin{align*}
\|\mathbf u^{(j+1)}-\mathbf u^*\|_2\le C\|\mathbf u^{(j)}-\mathbf u^*\|_2^2.
\end{align*}
\end{theorem}
\begin{proof}
We follow the proof of \cite[Theorem 3]{Pa91}. By assumption, we have $t_j=1$, i.e.\ $\mathbf u^{(j+1)}=\mathbf u^{(j)}+\mathbf d^{(j)}$, for all $j$ large enough. With $\mathcal B(\mathbf u^{(j)})$, $\mathcal C(\mathbf u^{(j)})$ from \eqref{eq:BC}, we have
\begin{align*}
F_k(\mathbf u^{(j)})+\gamma(\nabla^2 g(\mathbf u^{(j)})\mathbf d^{(j)})_k=0,\qquad & \text{for}\ k\in\mathcal B(\mathbf u^{(j)}),\\
u_k^{(j+1)}=0,\qquad & \text{for}\ k\in\mathcal C(\mathbf u^{(j)}).
\end{align*}
Because $\mathbf u^*$ is the limit of $\{\mathbf u^{(j)}\}_j$, we have for $j$ large enough
\begin{align*}
\mathcal A(\mathbf u^*)&\subset \big( \{k: u_k^{(j)}>\gamma(\nabla g(\mathbf u^{(j)}))_k+\gamma w_k \wedge u_k^{(j)}>0\}\\
& \phantom{\subset}\ \cup \{k: u_k^{(j)}<\gamma(\nabla g(\mathbf u^{(j)}))_k-\gamma w_k \wedge u_k^{(j)}<0\} \big)\\
&\subset \mathcal B(\mathbf u^{(j)}).
\end{align*}
This yields the inclusion $\mathcal C(\mathbf u^{(j)})\subset \mathcal I^+(\mathbf u^*)\cup\mathcal I^-(\mathbf u^*)\cup\mathcal I^\circ (\mathbf u^*)$, implying $\mathbf u^*_{\mathcal C(\mathbf u^{(j)})}=\mathbf 0$.
Analogously, we have for $j$ large enough
\begin{align*}
& \mathcal I^\circ (\mathbf u^*)\\
\subset\ & \{k: |\mathbf u_k^ {(j)}-\gamma (\nabla g(\mathbf u^{(j)}))_k|<\gamma w_k \wedge \gamma (\nabla g(\mathbf u^{(j)}))_k+\gamma w_k>0 \}\\
& \cup\ \{ k: |\mathbf u_k^{(j)}-\gamma (\nabla g(\mathbf u^{(j)}))_k|<\gamma w_k  \wedge \gamma (\nabla g(\mathbf u^{(j)}))_k-\gamma w_k<0 \}\\
 \subset\ & \mathcal C(\mathbf u^{(j)}).
\end{align*}
Consequently, we have $\mathcal B(\mathbf u^{(j)})\subset \mathcal I^+(\mathbf u^*)\cup\mathcal I^-(\mathbf u^*)\cup\mathcal A(\mathbf u^*)$, implying $0=F_k(\mathbf u^*)=\gamma\nabla g(\mathbf u^*)_k\pm \gamma w_k$, respectively, for all $k\in\mathcal B(\mathbf u^{(j)})$.

Skipping the arguments $\mathcal B=\mathcal B(\mathbf u^{(j)})$, $\mathcal C=\mathcal C(\mathbf u^{(j)})$, we obtain with $\mathbf u^*_\mathcal C=\mathbf 0$, $\mathbf F(\mathbf u^*)_\mathcal B=\mathbf 0$ and the mean value theorem
\begin{align*}
&\begin{pmatrix}\big(\mathbf G(\mathbf u^{(j)})(\mathbf u^{(j+1)}-\mathbf u^*)\big)_{\mathcal B}\\ \big(\mathbf G(\mathbf u^{(j)})(\mathbf u^{(j+1)}-\mathbf u^*)\big)_{\mathcal C}\end{pmatrix}\\
=&\begin{pmatrix}(\nabla^2 g(\mathbf u^ {(j)}))_{\mathcal B,\mathcal B} & (\nabla^2 g(\mathbf u^ {(j)}))_{\mathcal B,\mathcal C}\\
\mathbf 0_{\mathcal C,\mathcal B} & \mathbf I_{\mathcal C,\mathcal C}\end{pmatrix}\begin{pmatrix}(\mathbf u^{(j+1)}-\mathbf u^{(j)}+\mathbf u^{(j)}-\mathbf u^*)_{\mathcal B}\\(\mathbf u^{(j+1)}-\mathbf u^{(j)}+\mathbf u^{(j)}-\mathbf u^*)_{\mathcal C}\end{pmatrix}\\
=& \begin{pmatrix}
-\mathbf F(\mathbf u^{(j)})_{\mathcal B}+\gamma \nabla^2 g(\mathbf u^{(j)})_{\mathcal B,\mathcal B}(\mathbf u^{(j)}-\mathbf u^*)_{\mathcal B}+\gamma \nabla^2 g(\mathbf u^{(j)})_{\mathcal B,\mathcal C}(\mathbf u^{(j)}-\mathbf u^*)_{\mathcal C}\\
-\mathbf u^{(j)}_{\mathcal C}+(\mathbf u^{(j)}-\mathbf u^*)_{\mathcal C}
\end{pmatrix}\\
&+\begin{pmatrix}
\mathbf F(\mathbf u^*)_{\mathcal B}\\{\mathbf u}^*_{\mathcal C}
\end{pmatrix}\\
=& \begin{pmatrix}
\Big(\sum\limits_{l,m=1}^{n}\gamma\frac{\partial^3 g(\mathbf v)}{\partial u_l\partial u_m\partial u_k}(u_l^*-u_l^{(j)})(u_m^*-u_m^{(j)})\Big)_{k\in\mathcal B}\\ \mathbf 0_{\mathcal C}
\end{pmatrix},
\end{align*}
where $\mathbf v$ is a vector on the line segment between $\mathbf u^{(j)}$ and $\mathbf u^*$.
For $j$ large enough, the matrix $\mathbf G(\mathbf u^{(j)})$ is boundedly invertible by Assumption \ref{ass:1}, cf. \eqref{eq:estimateinverse}. Therefore, there exists a constant $C>0$, depending only on $\mathbf u^*$, with
\begin{align*}
\|\mathbf u^{(j+1)}-\mathbf u^*\|_2\le C \|\mathbf u^{(j)}-\mathbf u^*\|_2^2,
\end{align*}
for all $j$ large enough, proving the claim.
 \end{proof}

Note that in case of a quadratic functional $g(\mathbf u)=\frac{1}{2}\|\mathbf K\mathbf u-\mathbf f\|_2^2$ with $\mathbf K$ injective, $\mathbf G(\mathbf u)^{-1}$ was shown to be uniformly bounded in a neighborhood of the zero $\mathbf u^*$ of $\mathbf F$ \cite{GrLo08}. Hence, in case of a quadratic functional $g$ with $\mathcal I^+(\mathbf u^*)\cup\mathcal I^-(\mathbf u^*)=\emptyset$, the stepsizes in Algorithm \texttt{modBSSN} are eventually chosen equal to $1$, locally quadratic convergence is achieved and $\mathbf u^*$ is found within finitely many steps, see also Remark \ref{remark:indexsets} and \cite{HaRa15}. For other functionals $g$, these conditions need to be verified. 

\subsection{Convergence of the B-semismooth Newton method}\label{sec:32}
 In this section, we consider Algorithm \texttt{BSSN}, i.e.\ the B-semismooth Newton method from \cite{HaRa15} generalized to the minimization problem \eqref{eq:min}, see Section \ref{sec:22}. We cite the convergence theorem from \cite[Theorem 4.8]{HaRa15}, see also \cite[Theorem 1]{HaPaRa92}.
\begin{theorem}\label{th:convergenceoldSSN}
Let Assumption \ref{ass:1} be fulfilled and let $\{\mathbf u^{(j)}\}_j$ be a sequence of iterates produced by Algorithm $\mathtt{BSSN}$ from Section \ref{sec:22}. 
Let $\{t_j\}_j$ be the chosen stepsizes. 
\begin{itemize}
\item [(i)] If $\lim\sup_{j\to\infty}t_j>0$, then $\mathbf u^{(j)}\to\mathbf u^*$, $j\to\infty$ with $\Theta(\mathbf u^*)=0$.
\item [(ii)] If $\lim\sup_{j\to\infty}t_j=0$ and if $\mathbf u^*$ is an accumulation point of $\{\mathbf u^{(j)}\}_j$, where condition \eqref{eq:convergenceconditionSSN} holds at $\mathbf u^*$, then $\mathbf u^{(j)}\to\mathbf u^*$, $j\to\infty$ with $\Theta(\mathbf u^*)=0$.
\end{itemize}
\end{theorem}
\begin{proof}
The proof follows \cite[Proof of Theorem 4.8]{HaRa15} using Assumption \ref{ass:1}, Lemma \ref{lemma:dunique}, Lemma \ref{lemma:armijo} and Lemma \ref{lemma:dbounded}.
\end{proof}

Analogously to \cite[Corollary 4.10]{HaRa15}, we can deduce from \cite[Theorem 4.3, Corollary 4.4]{Qi93} that if the zero $\mathbf u^*$ of $\mathbf F$ is an accumulation point of a sequence $\{\mathbf u^{(j)}\}_j$ of iterates produced by Algorithm \texttt{BSSN}, the sequence $\{\mathbf u^{(j)}\}_j$ converges locally superlinearly to $\mathbf u^*$ and the stepsizes $t_k$ are eventually chosen equal to $1$. Nevertheless, the modification of the index sets is essential for the modified B-semismooth Newton method (Algorithm \texttt{modBSSN}) to overcome the theoretical drawback of the technical assumption \eqref{eq:convergenceconditionSSN} in Theorem \ref{th:convergenceoldSSN}, see Section \ref{sec:31}.

\subsection{Convergence of the hybrid method}
The global convergence and the local convergence speed of Algorithm \texttt{hybridBSSN} from Section \ref{sec:23} directly follow from Theorem \ref{th:convergence} and Theorem \ref{th:quadraticconvergence} resp.\ Section \ref{sec:32}. The method combines the efficiency of Algorithm \texttt{BSSN} and the stronger convergence properties of Algorithm \texttt{modBSSN}.

\section{Numerical results}\label{sec:numericalexperiments}
In this section, we present numerical experiments demonstrating our theoretical results. We first consider image deblurring for gray-scale images degraded by motion blur. This is a linear inverse problem and in the presence of noisy measurement data regularization is essential. Assuming that the image is sparse, i.e.\ it has only few nonzero pixels, we apply $\ell_1$-penalized Tikhonov regularization, compare \eqref{eq:min_quadraticg}. Here, Assumption \ref{ass:1} is fulfilled. Second, we consider a nonquadratic functional $g$ arising in robust linear regression. If data is degraded by outliers, instead of minimizing the ordinary least squares functional one may choose a more robust objective function, see e.g.\ \cite{Al11,ClWo15,Fu99}. Giving preference to simple models, we add a sparsity promoting penalty term as proposed in current research effecting that irrelevant coefficients are set equal to zero, see e.g.\ \cite{AlCrGe13,LiScRaCe15,Ti96} and the references therein. For the arising minimization problem \eqref{eq:min}, it is not ensured that all prior assumptions are fulfilled. Nevertheless, convincing numerical results are achieved.

For our numerical experiments, we use MATLAB$^\circledR$ 2015a and the computations are run on a desktop PC with Intel$^\circledR$ Xeon$^\circledR$ CPU  (W3530, 2.80 GHz). In Algorithm \texttt{modBSSN}, Algorithm \texttt{BSSN} and Algorithm \texttt{hybridBSSN}, see Algorithm \ref{algo1}, we choose the Armijo parameters $\sigma=0.01$ and $\beta=0.5$. The stopping criterion is a residual norm $\|\mathbf F(\mathbf u^{(j)})\|_2$ smaller than $10^{-7}$ in all computations. If not otherwise stated, the zero vector is chosen as starting vector. In Algorithm \texttt{hybridSSN}, we choose $j_{max}=250$ and $t_{min}=10^{-5}$.

The performance of Algorithm \texttt{modBSSN}, Algorithm \texttt{BSSN} and Algorithm \texttt{hybridBSSN} depends on the choice of the parameter $\gamma$ as well as, at least concerning Algorithm \texttt{modBSSN}, the particular solver for the linear complementarity problem \eqref{eq:LCP}. In our numerical experiments, the linear complementarity problem is solved with the modified damped Newton method from \cite{HaPa90}. This algorithm is a specialization of the method from \cite{Pa90} to linear complementarity problems. It was shown in \cite{FiKa96} that the method finds the true solution to the linear complementarity problem within finitely many iterations. The stopping criterion for an iterate $\tilde{\mathbf x}$ is here chosen as $\|\min\{\tilde{\mathbf x},\mathbf z+\mathbf N\tilde{\mathbf x}\}\|_2<10^{-7}$. If the starting vector $\mathbf x^{(0)}\in\mathbb R^{|\overline{\mathcal I^\pm}|}$ fulfills $(\mathbf z+\mathbf N\mathbf x^{(0)})_k\neq x_k^{(0)}$ for all $k$ where $\mathbf N$, $\mathbf z$ from \eqref{eq:BSSNN} resp.\ \eqref{eq:BSSNz}, which is the case if e.g.\ $\mathbf x^{(0)}:=\mathbf 0$ and if $z_k\neq 0$ for all $k$, the Newton method only poses one linear system per iteration \cite{HaPa90}. We choose $\mathbf x^{(0)}:=\mathbf 0$. If this condition is violated by the starting vector or if more than $50$ Newton steps are needed, we switch to an implementation\footnote[1]{The code is taken from \texttt{http://code.google.com/p/rpi-matlab-simulator/source/browse/}\\\texttt{simulator/engine/solvers/Lemke/lemke.m} (30 June 2015).} of Lemke's algorithm \cite{CoPaSt09,WiLuetal13}. The damped Newton method from \cite{HaPa90} is often faster than Lemke's method in terms of computational time, see also the numerical results in \cite{HaPa90}. We also tested an interior point method using the MATLAB function \texttt{quadprog} and an implementation\footnote[2]{The code is taken from \texttt{http://www.mathworks.com/matlabcentral/fileexchange/}\\ \texttt{20952-lcp---mcp-solver--newton-based-/content/LCP.m} (30 June 2015).} of the semismooth Newton-type method \cite{Fi95} based on a Fischer-Burmeister reformulation of the linear complementarity problem as well as the PATH solver\footnote[3]{The code is taken from \texttt{http://pages.cs.wisc.edu/$\sim$ferris/path.html} (08 February 2016).} from \cite{DiFe95,FeMu99}. We decided to solve the linear complementarity problem up to machine precision because its inexact solution may cause an increased number of Newton steps. The arising systems of linear equations are solved with a direct solver (MATLAB backslash subroutine).

\subsection{Image deblurring}

We consider the deblurring of images which are degraded by horizontal motion blur caused by either motion of the camera or the photographed object while taking a photo. Here, we proceed as in \cite{Ha02}. 
Our aim is the reconstruction of the original square image $\mathbf u$ from noisy measurements of the blurred image $\mathbf f$. 
As proposed in \cite{Ha02}, we consider the discrete problem $\mathbf K\mathbf u=\mathbf f$,
where $\mathbf u,\mathbf f\in\mathbb R^{N^2}$ and the Toeplitz matrix
\begin{equation}\label{eq:Kblur}
\mathbf K=\frac{1}{2\lfloor NL\rfloor+1} \begin{pmatrix}1 & \cdots & 1 & 0 & \cdots & 0\\ \vdots&\ddots& & \ddots & \ddots & \vdots\\
1 & & \ddots & & \ddots &0\\
0 & \ddots & & \ddots & & 1\\
\vdots & & \ddots & & \ddots & \vdots\\
0 & \cdots & 0 & 1 & \cdots& 1\end{pmatrix}\otimes \mathbf I\in\mathbb R^{N^2\times N^2},
\end{equation}
where the matrix on the left-hand side of the Kronecker product has bandwidth $2\lfloor NL\rfloor+1$ and where $\mathbf I\in\mathbb R^{N\times N}$ denotes the identity matrix. The blurring parameter $L$ characterizes the motion blurring of the image and we choose $L=0.1$. To avoid inverse crime, we discretize the problem with the Simpson rule to compute the blurred image $\mathbf f$ and use the discretization \eqref{eq:Kblur} to solve the inverse problem. The noise is computed with the MATLAB function \texttt{randn} and the noisy blurred image $\mathbf f^\delta$ contains $5\%$ relative noise, i.e.\ we have $\|\mathbf f-\mathbf f^\delta\|_2=5\% \|\mathbf f\|_2$.

The regularization parameters $w_k=w$, $k=1,\ldots,N^2$ are chosen equal and $w$ is computed by the discrepancy principle, see e.g.\ \cite{AnRa10,Bo09,EnHaNe96, ScKaHoKa12}. More precisely, we choose $w=0.9^{10}$, $q=0.9$ and $\tau=2$ and set $w:=wq$ until the inequality $\|\mathbf K\mathbf u_w-\mathbf f^\delta\|_2\le\tau \|\mathbf f-\mathbf f^\delta\|_2$ is fulfilled, where $\mathbf u_w$ denotes the solution to \eqref{eq:min} with $w_k=w$ for all $k$, $n=N^2$ and $g(\mathbf u)=\frac 12 \|\mathbf K\mathbf u-\mathbf f^\delta\|_2^2$. For each computation of $\mathbf u_w$, we choose the minimizer $\mathbf u_{\tilde w}$ of the Tikhonov functional with $\tilde w=w/0.9$ as starting vector. In this subsection, we mainly consider Algorithm \texttt{modBSSN} because the performance of \texttt{BSSN} from Section \ref{sec:22}  for quadratic functionals $g$ was discussed in \cite{HaRa15}.

\begin{figure}[t]
\begin{center}
\includegraphics[scale=0.25]{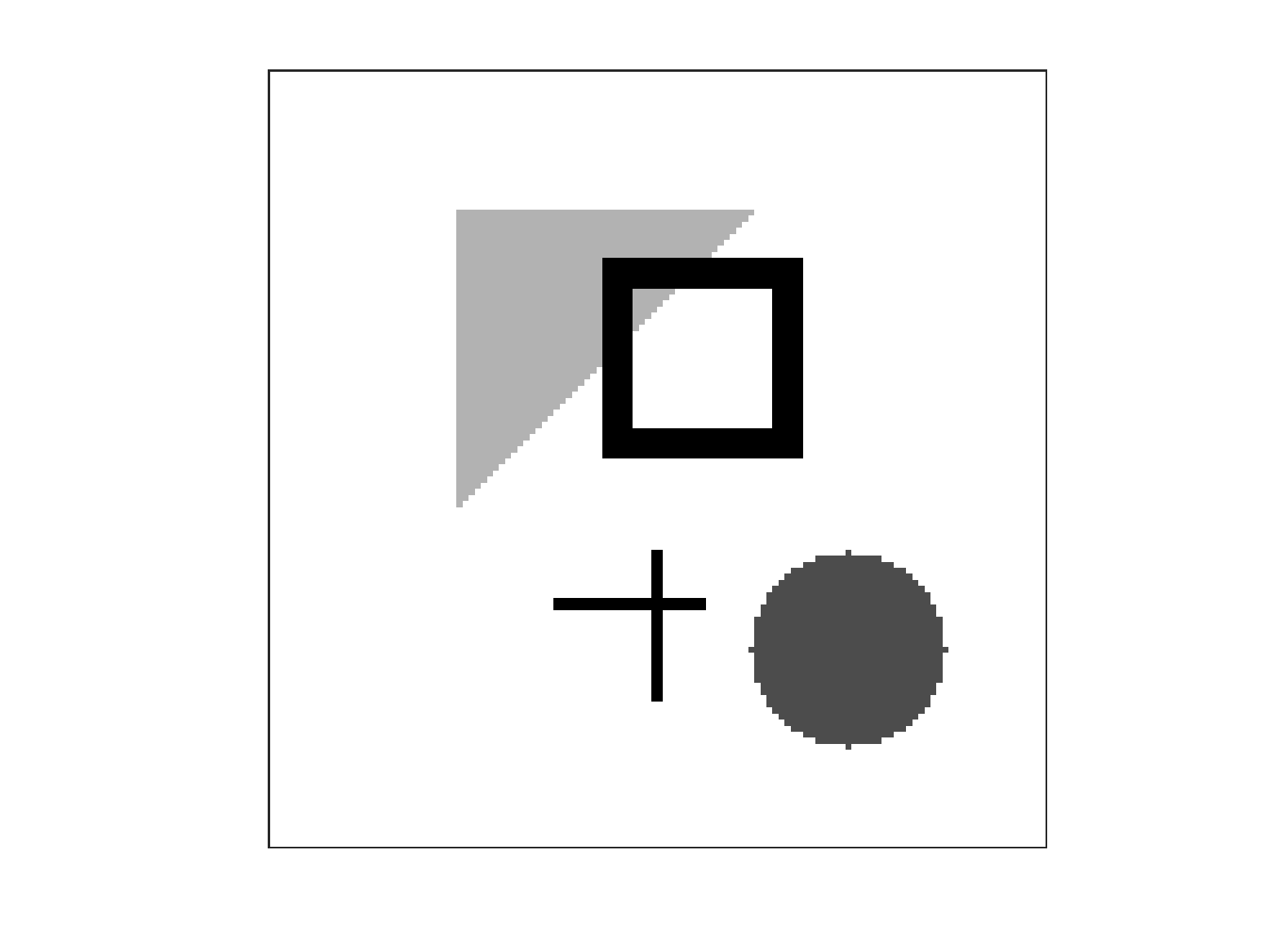}
\includegraphics[scale=0.25]{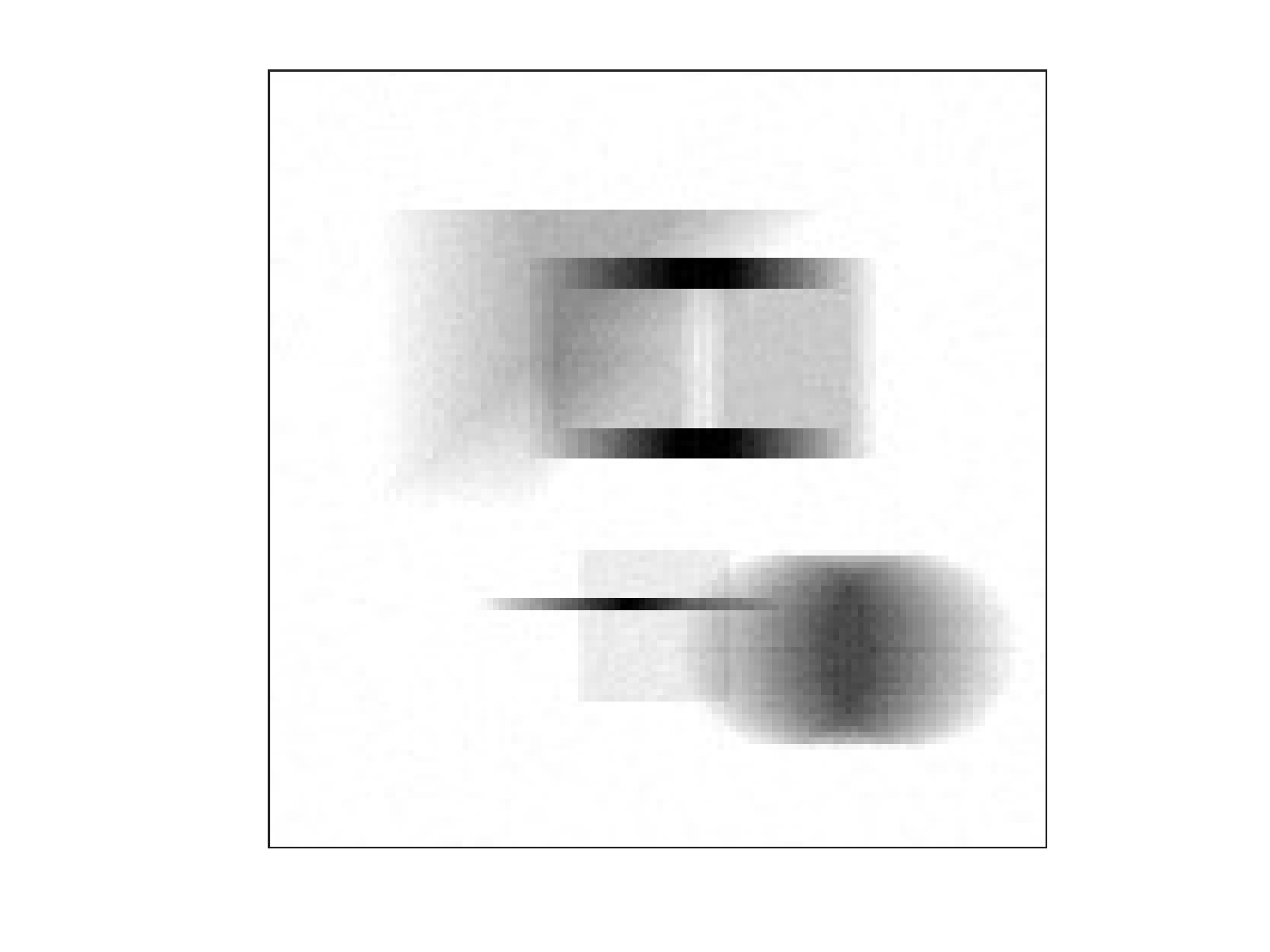}
\includegraphics[scale=0.25]{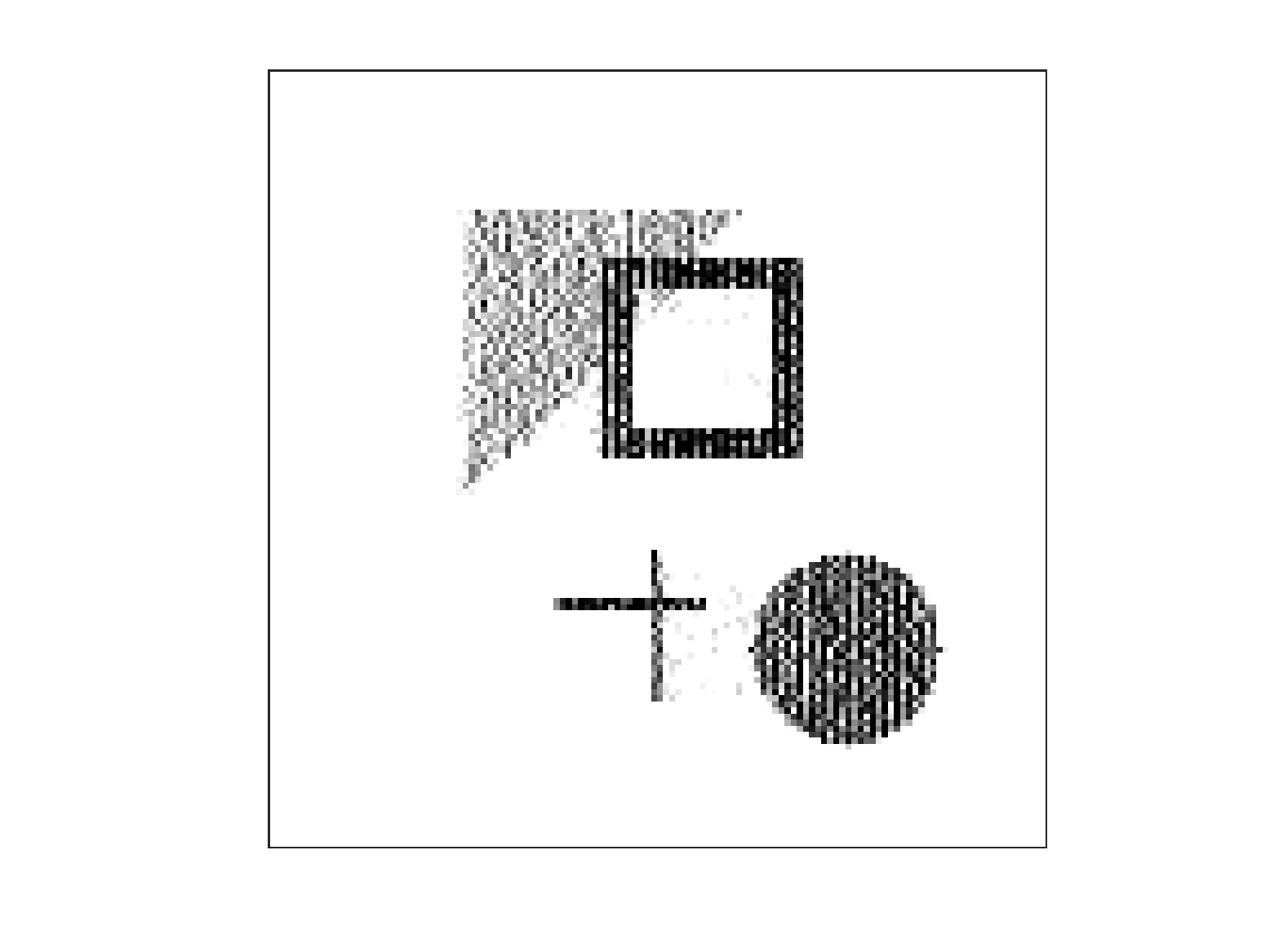}
\end{center}
\caption{Reconstruction of a blurred test image with $N^2=128^2$ pixels containing $5\%$ of noise using Algorithm \texttt{modBSSN} with regularization parameter $w=0.9^{33}\approx 0.0309$, $\gamma=10^5$ and blurring parameter $L=0.1$. From left to right: original image, blurred noisy image, reconstruction.}\label{fig:1}
\end{figure}

We consider an artificially created sparse image with about $15\%$ nonzero entries, the sparseness depends on the number $N^2$ of pixels, see Figure \ref{fig:1}. Here, the original image of the size $128\times 128$ pixels, the blurred image containing $5\%$ of noise and the reconstruction are presented. The blurring parameter is chosen as $L=0.1$, the regularization parameter is chosen as $w=0.9^{33}\approx 0.0309$ and the parameter $\gamma$ in Algorithm \texttt{modBSSN} is set equal to $10^5$. 

\begin{table}[tb]\caption{Performance of Algorithm $\mathtt{modBSSN}$ depending on the choice of $\gamma$. The reconstructions are computed for the image from Figure \ref{fig:1} with $128^2=16384$ pixels containing $5\%$ of noise, regularization parameter $w=0.9^{33}\approx 0.0309$ and blurring parameter $L=0.1$.}\label{tab:1}
\small
\vspace{0.2cm}
\begin{tabular}{llll}
\hline\noalign{\smallskip}
 $\gamma$ & number of steps & $\#\{j: t_j=1\}$\\
\noalign{\smallskip}\hline\noalign{\smallskip}
$\texttt{10}^\texttt{1}$ & \texttt{113}   & \texttt{3}\\
$\texttt{10}^\texttt{2}$ & \texttt{23}   & \texttt{5}\\
$\texttt{10}^\texttt{3}$ & \texttt{12}   & \texttt{9}\\
$\texttt{10}^\texttt{4}$ & \texttt{12}   & \texttt{10}\\
$\texttt{10}^\texttt{5}$ & \texttt{11}   & \texttt{8}\\
$\texttt{10}^\texttt{6}$ & \texttt{11}   & \texttt{8}\\
$\texttt{10}^\texttt{7}$ & \texttt{13}   & \texttt{7}\\
\noalign{\smallskip}\hline
\end{tabular}
\end{table}

Table \ref{tab:1} demonstrates the performance of Algorithm \texttt{modBSSN} for the image from Figure \ref{fig:1} depending on the choice of $\gamma>0$. The parameter $\gamma$ should not be chosen too small, because the number of Newton steps increases and the amount of steps with stepsize $t_k=1$ decreases for smaller $\gamma$. For $\gamma=10^4$, the stepsizes are chosen equal to $1$ in $10$ out of $12$ steps.

\begin{table}[tb]
\caption{History of the residual norms, the Tikhonov functional values, the stepsizes, the system sizes of the linear complementarity problems (LCP) and of the systems of linear equations (SLE) and the number of linear systems for the image from Figure \ref{fig:1} with $128^2=16384$ pixels containing $5\%$ of noise, reconstructed with Algorithm \texttt{modBSSN} with regularization parameter $w=0.9^{33}\approx 0.0309$, the parameter $\gamma=10^5$ and $L=0.1$.}\label{tab:2}
\small
\vspace{0.2cm}
\begin{tabular}{lllllll}
\hline\noalign{\smallskip}
 $j$ & $\|\mathbf F(\mathbf u^{(j)})\|_2$ & $J(\mathbf u^{(j)})$ & $t_j$ & size of LCP & size of SLE & $\#$ SLE\\
\noalign{\smallskip}\hline\noalign{\smallskip}
\texttt{0} & \texttt{2.3200e+06}  & \texttt{357.1522} & - & - & - & -\\
\texttt{1} & \texttt{8.8461e+02}  & \texttt{105.8198} & \texttt{1} & \texttt{0} & \texttt{6043} & \texttt{1}\\
\texttt{2} & \texttt{7.9131e+02}  & \texttt{76.3797} & \texttt{1} & \texttt{616} & \texttt{3944} & \texttt{12441}\\
\texttt{3} & \texttt{5.7716e+02}  & \texttt{66.8937} & \texttt{0.5} & \texttt{441} & \texttt{3161} & \texttt{13224}\\
\texttt{4} & \texttt{4.2644e+02} & \texttt{59.0065} & \texttt{0.5} & \texttt{321} & \texttt{2769} & \texttt{13616}\\
\texttt{5} & \texttt{2.3887e+02} & \texttt{51.8326} & \texttt{1} & \texttt{220} & \texttt{2574} & \texttt{13811}\\
\texttt{6} & \texttt{2.0991e+02} & \texttt{49.7913} & \texttt{0.5} & \texttt{90} & \texttt{2452} & \texttt{13933}\\
\texttt{7} & \texttt{1.8523e+02} & \texttt{47.4883} & \texttt{1} & \texttt{67} & \texttt{2376} & \texttt{14009}\\
\texttt{8} & \texttt{4.8111e+01} & \texttt{46.7994} & \texttt{1} & \texttt{22} & \texttt{2357} & \texttt{14028}\\
\texttt{9} & \texttt{4.1728e-01}  & \texttt{46.4408} & \texttt{1} & \texttt{13} & \texttt{2330} & \texttt{14055}\\
\texttt{10} & \texttt{2.4710e-01}  & \texttt{46.4212} & \texttt{1} & \texttt{0} & \texttt{2326} & \texttt{1}\\
\texttt{11} & \texttt{4.5635e-10}  & \texttt{46.4061} & \texttt{1} & \texttt{0} & \texttt{2325} & \texttt{1}\\
\noalign{\smallskip}\hline
\end{tabular}
\end{table}

The strict decrease of the residual norm in Algorithm \texttt{modBSSN} for the image from Figure \ref{fig:1} with $128^2=16384$ pixels is demonstrated in Table \ref{tab:2}. Here, the Tikhonov functional values
\begin{equation}\nonumber
J(\mathbf u^{(j)}):=\frac{1}{2}\|\mathbf K\mathbf u^{(j)}-\mathbf f^\delta\|_2^2+w\|\mathbf u^{(j)}\|_1,\quad j=0,1,\ldots
\end{equation}
are strictly decreasing as well, but this is not guaranteed in general. The stepsizes are eventually chosen equal to $1$ ensuring the locally quadratic convergence of Algorithm  \texttt{modBSSN}. The sizes of the linear complementarity problem (LCP), see \eqref{eq:LCP}, and of the systems of linear equations (SLE), solved in each step of Algorithm \texttt{modBSSN} to compute the matrix $\mathbf N$ from \eqref{eq:BSSNN}, the vector $\mathbf z$ from \eqref{eq:BSSNz} and the Newton direction \eqref{eq:d}, are usually decreasing in the course of the iteration. Regarding the number $128^2=16384$ of pixels, these systems are small. This is due to the structure of Algorithm \texttt{modBSSN}. Because of the starting vector $\mathbf u^{(0)}=\mathbf 0$, the set $\overline{\mathcal I^\pm}(\mathbf u^{(0)})$ is usually empty so that there is usually no LCP to solve in the first step. For other starting vectors $\mathbf u^{(0)}$, the size of the LCP in the first step may be larger. If a linear complementarity problem is set up in step $j$, additionally $|\overline{\mathcal I}(\mathbf u^{(j)})|$ linear systems with the same matrix have to be solved, cf.\ Section \ref{sec:23}.

In Table \ref{tab:3}, five algorithms for the deblurring of the noisy image from Figure \ref{fig:1} are compared: Algorithm \texttt{modBSSN} and Algorithm \texttt{hybridBSSN} with the choice $\gamma=10^5$, the globalized semismooth Newton method (\texttt{BSSN}) from \cite{HaRa15} with the choice $\gamma=10^5$, sparse reconstruction by separable approximation\footnotemark[1] (\texttt{SpaRSA}) from \cite{WrNoFi09} and Barzilai-Borwein gradient projection for sparse reconstruction\footnote[1]{The implementations of \texttt{SpaRSA}, and \texttt{GPSR$\_$BB} are taken from \texttt{http://www.lx.it.pt/$\sim$mtf/SpaRSA/}\\ and \texttt{http://www.lx.it.pt/$\sim$mtf/GPSR/} respectively (30 June 2015).} (\texttt{GPSR$\_$BB}) from \cite{FiNoWr07}. Note that runtime is implementation-dependent. 
Note also that \texttt{BSSN} differs from \texttt{modBSSN} only in the choice of the index sets \eqref{eq:A+}--\eqref{eq:I-} resp.\ the modified index sets \eqref{eq:A+modified}--\eqref{eq:I-modified}. By the modification of the index sets, the theoretical drawback that \texttt{BSSN} may fail to converge was eliminated. Therefore, one has to solve mixed linear complementarity problems instead of solving only systems of linear equations. In practice, applying \texttt{BSSN}, complementarity problems usually do not appear. The stopping criterion of Algorithm \texttt{modBSSN}, Algorithm \texttt{hybridBSSN} and Algorithm \texttt{BSSN} is a residual norm $\|\mathbf F(\mathbf u^{(j)})\|_2< 10^{-7}$. The other three algorithms are terminated if the Tikhonov functional value falls below the threshold $J^*+10^{-7}$, where $J^*$ denotes the Tikhonov functional value of \texttt{hybridBSSN} at convergence. The average runtime (clock time) of five runs with starting vector $\mathbf u^{(0)}=\mathbf 0$, the Tikhonov functional value $J_{end}$ at termination, the difference of $J_{end}$ to $J^*$, the number of iterations and the number of zeros of the computed solution are listed for the different algorithms. All algorithms produce sparse solutions with $14059$ resp.\ $14057$ zero components, i.e.\ about $14.2\%$ nonzero entries. The semismooth Newton methods need only few iterations compared to the other methods. The fastest algorithms are \texttt{BSSN} and \texttt{hybridBSSN} followed by \texttt{SpaRSA}, \texttt{modBSSN} and \texttt{GPSR$\_$BB}. The runtime of Algorithm \texttt{modBSSN} may be improved by using another solver for the linear complementarity problems. In Table \ref{tab:3} and in the following runtime measurements, the computation of the regularization parameter by the discrepancy principle is not included in the listed runtimes. The runtimes are measured with the MATLAB command \texttt{tic toc}. 

\begin{table}[tbp]
\caption{Comparison of different algorithms for the deblurring of the image from Figure \ref{fig:1} with $N=128^2$ pixels, $5\%$ of noise, blurring parameter $L=0.1$ and regularization parameter $w=0.9^{33}\approx 0.0309$. The starting vector $\mathbf u^{(0)}=\mathbf 0$ is chosen for all algorithms.}\label{tab:3}
\small
\vspace{0.2cm}
\begin{tabular}{l|lllll}
\hline\noalign{\smallskip}
algorithm & average runtime(s) & $J_{end}$ & $J_{end}-J^*$ & $\#$ iterations & $\#$ zeros\\
\noalign{\smallskip}\hline\noalign{\smallskip}
  \texttt{modBSSN} & \texttt{8.54} & \texttt{46.4061} & \texttt{1.4211e-14} & \texttt{11}& \texttt{14059}\\ 
  \texttt{BSSN} & \texttt{0.32} & \texttt{46.4061} & \texttt{0} & \texttt{13}& \texttt{14059}\\
   \texttt{hybridBSSN} & \texttt{0.32} & \texttt{46.4061} & \texttt{0} & \texttt{13}& \texttt{14059}\\
  \texttt{SpaRSA} & \texttt{1.96} & \texttt{46.4061} & \texttt{9.6221e-08} & \texttt{1057}& \texttt{14057}\\ 
  \texttt{GPSR$\_$BB} & \texttt{15.21} & \texttt{46.4061} & \texttt{9.9948e-08} &\texttt{6352} & \texttt{14059}\\
\noalign{\smallskip}\hline
\end{tabular}
\end{table}

\begin{figure}[t]
\begin{center}
\includegraphics[scale=0.31]{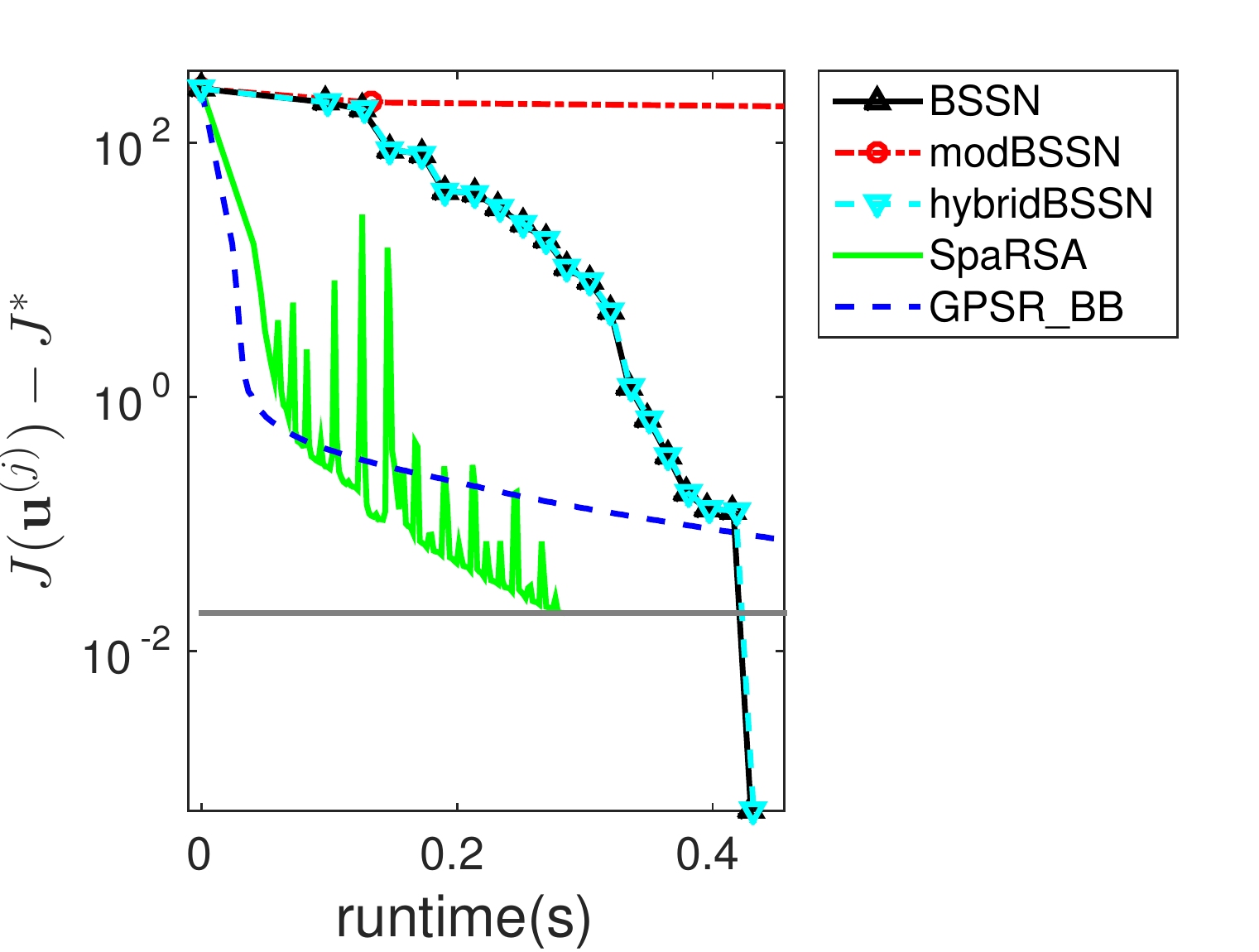}
\includegraphics[scale=0.31]{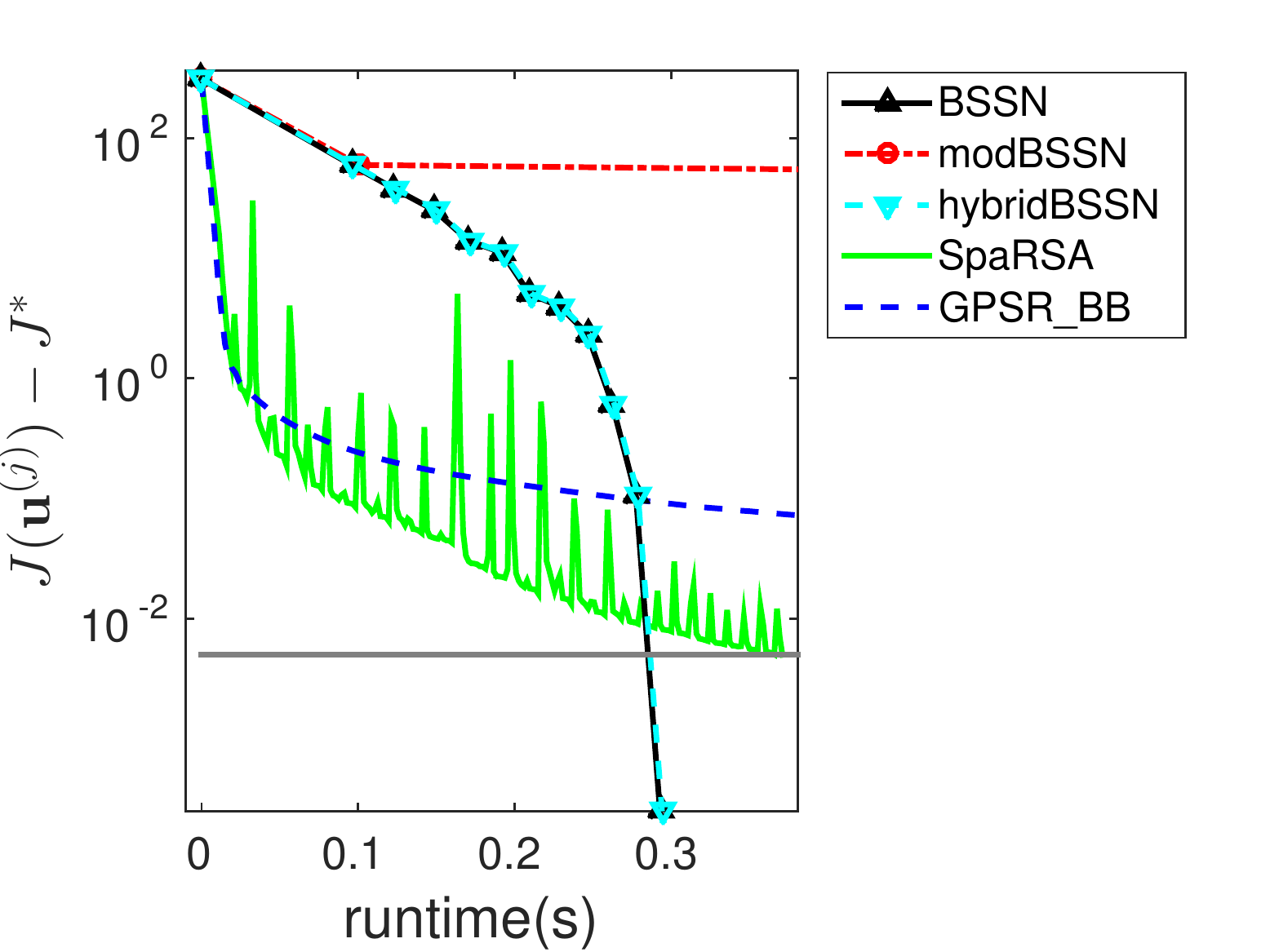}
\includegraphics[scale=0.31]{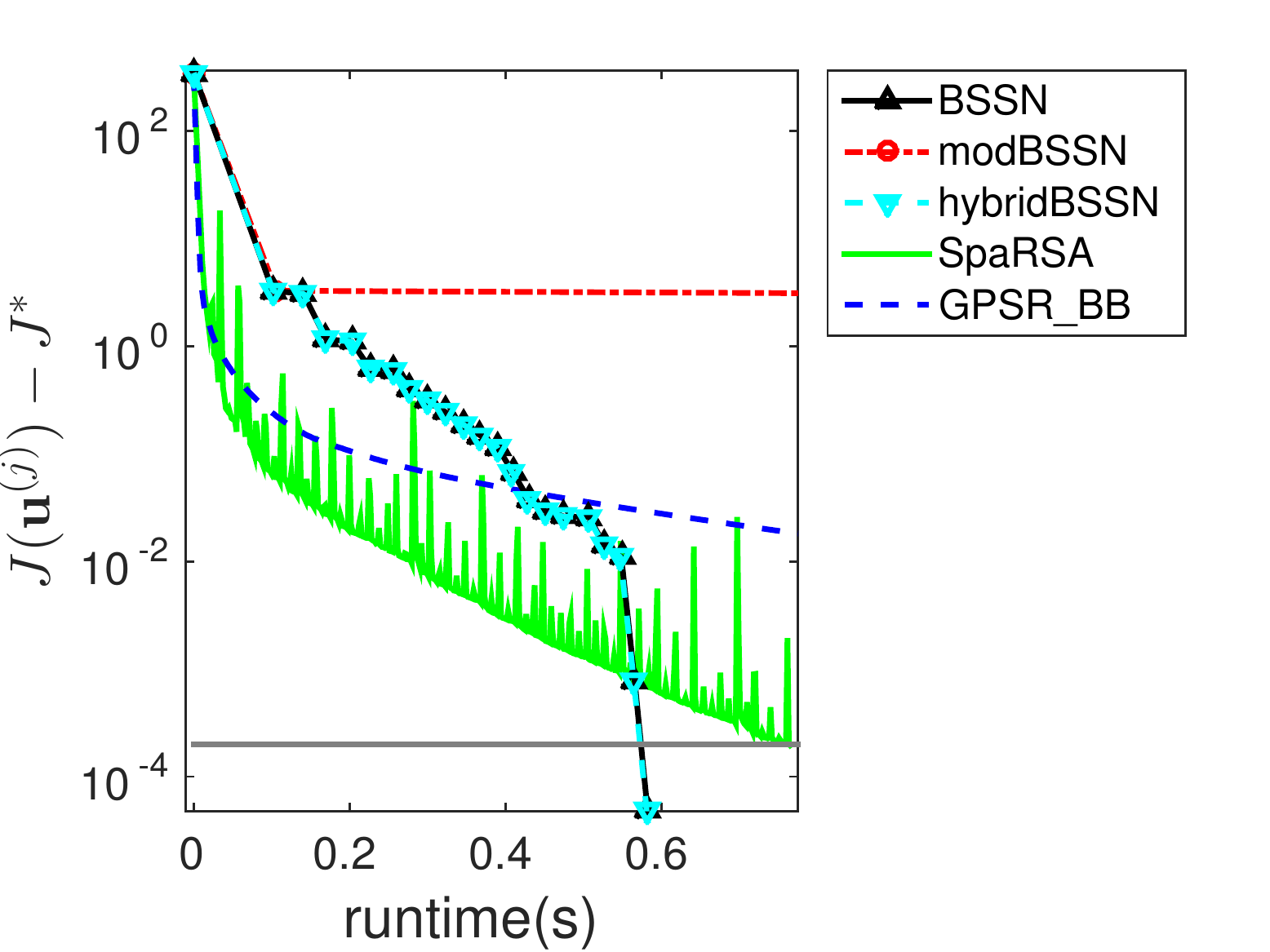}
\end{center}
\caption{Runtime history of the difference of the Tikhonov functional values and $J^*$ for different noise levels $\delta$. From left to right: $\delta=10\%$, $5\%$, $1\%$. The gray line marks the target $2\delta^2$ in each case.}\label{fig:2}
\end{figure}

The runtime history of the difference $J(\mathbf u^{(j)})-J^*$ of the Tikhonov functional values $J(\mathbf u^{(j)})$ of the algorithms considered in Table \ref{tab:3} to the Tikhonov functional value $J^*$ of Algorithm \texttt{hybridBSSN} at convergence is shown in Figure \ref{fig:2} for different noise levels $\delta=10\%$, $5\%$, $1\%$. The parameter $\gamma$ in the algorithms \texttt{BSSN}, \texttt{modBSSN} and \texttt{hybridBSSN} is chosen equal to $10^5$ and we set $L=0.1$ and $N^2=128^2$. Depending on the noise level, it may be adequate to solve the minimization problem only up to an expected accuracy. $\ell_1$-Tikhonov regularization with a posteriori parameter choice by the discrepancy principle has a linear convergence rate, see \cite{GrHaSc11}, i.e.\ $\|\mathbf u^\dagger-\mathbf u^*_{w,\delta}\|\le c\delta\|\mathbf f\|$, where $c>0$ is a constant, $\mathbf u^\dagger$ denotes the true solution to $\mathbf K\mathbf u=\mathbf f$ with unperturbed right-hand side $\mathbf f$ and $\mathbf u^*_{w.\delta}$ denotes the solution to \eqref{eq:min_quadraticg} with perturbed data $\mathbf f^\delta$, regularization parameter $w$ and noiselevel $\delta$. Therefore, we decided to minimize the Tikhonov functional up to an accuracy of $2\delta^2$. For high noise levels and $N^2=128^2$, Algorithm \texttt{SpaRSA} outperforms \texttt{BSSN} and \texttt{hybridBSSN} because it reaches the target first. If the minimization problem is solved more accurately in case of smaller noise levels or if the number $N^2$ of pixels increases, \texttt{BSSN} and \texttt{hybridBSSN} are advantageous in terms of runtime in this example, cf.\ Figure \ref{fig:3}.

\begin{figure}[t]
\begin{center}
\includegraphics[scale=0.35]{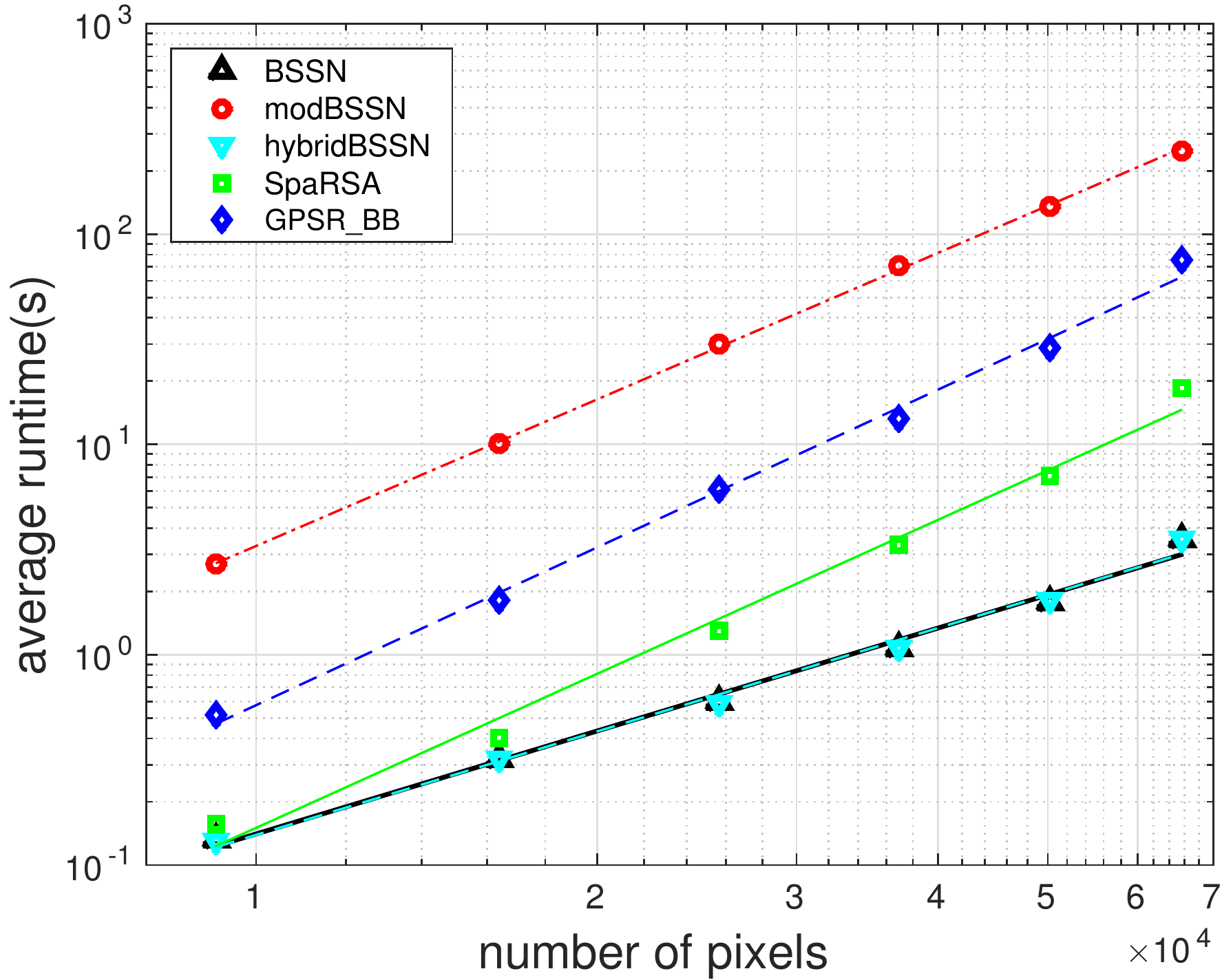}\quad
\includegraphics[scale=0.35]{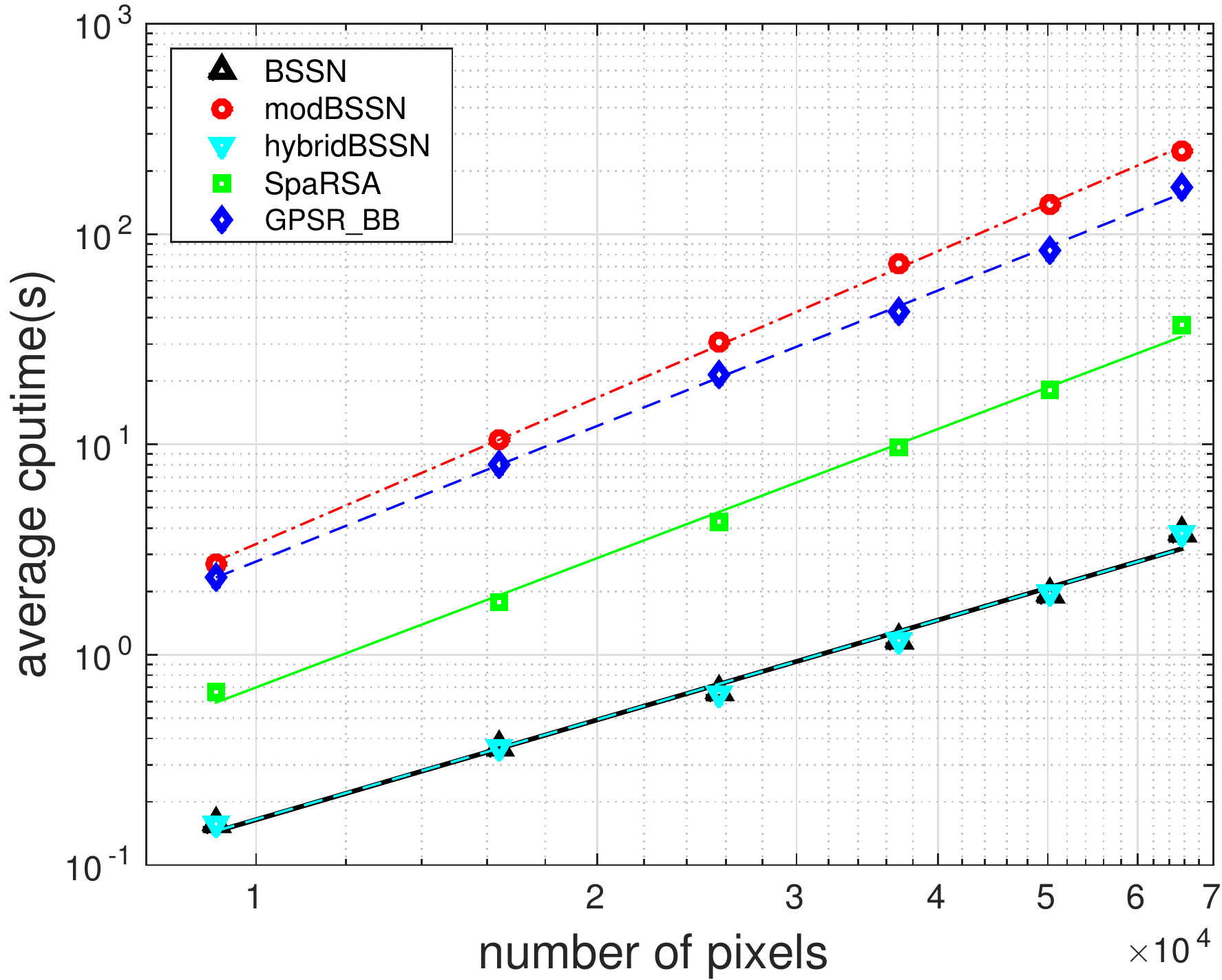}
\end{center}
\caption{Average runtime and average cputime of $5$ runs depending on the number $N^2=(32k)^2$, $k=3,\ldots,8$ of pixels for images containing $5\%$ of noise.}\label{fig:3}
\end{figure}

Figure \ref{fig:3} presents a clock time and a cputime comparison of the considered algorithms for increasing image sizes $N^2=(32k)^2$, $k=3,\ldots,8$. The cputime is measured with the MATLAB subroutine \texttt{cputime}. Once again, the blurring parameter is $L=0.1$ and the images contain $5\%$ of noise. The starting vector is $\mathbf u^{(0)}=\mathbf 0$ for all methods, the stopping criterion $J(\mathbf u^{(j)})\le J^*+2\delta^2$ for \texttt{GPSR$\_$BB} and \texttt{SpaRSA} is chosen as in Figure \ref{fig:2} and we choose $\gamma=10^4$ and the stopping criterion $\|\mathbf F(\mathbf u^{(j)})\|_2< 10^{-7}$ for \texttt{BSSN}, \texttt{hybridBSSN} and \texttt{modBSSN}. Again, the average runtimes resp.\ cputimes of $5$ runs are shown. Algorithms \texttt{BSSN} and \texttt{hybridBSSN} outperform the other algorithms regarding cputime in this example, followed by \texttt{SpaRSA}, \texttt{GPSR$\_$BB} and \texttt{modBSSN}. However, \texttt{SpaRSA} and \texttt{GPSR$\_$BB} are better parallelizable than the B-semismooth Newton methods.

\subsection{Robust regression}
Given data $\mathbf a_1,\ldots,\mathbf a_m\in\mathbb R^n$ and $\mathbf y\in\mathbb R^m$, $m\ge n$, our aim is to fit a linear model $\mathbf A\mathbf u=\mathbf y$ with $\mathbf u\in\mathbb R^n$ and $\mathbf A=(\mathbf a_1\cdots \mathbf a_m)^\top\in\mathbb R^{m\times n}$ to the given data. Errors in data collection may cause outliers, and robust M-estimators give less influence to outliers than the ordinary least squares approach \cite{Fu99}. Here, we choose the well-known $L_1$-$L_2$ estimator, see e.g.\ \cite{ClWo15}.
For a parameter $\rho>0$, the measure function $\varphi_\rho\colon\mathbb R\to\mathbb R^+_0$, $\varphi_\rho(x):=2(\sqrt{\rho+x^2/2}-\sqrt{\rho})$ fulfills the conditions $\varphi_\rho(x)=\varphi_\rho(-x)$ and $\varphi_\rho$ is strictly convex \cite{Al11, ClWo15}. We choose $\rho=1$ and the discrepancy term $g\colon\mathbb R^n\to\mathbb R$,
\begin{equation}\nonumber
g(\mathbf u):=\frac 1m\sum\limits_{k=1}^m \varphi_1(\mathbf a_k^\top\mathbf u-y_k)=\frac 2m\sum\limits_{k=1}^{m} \Big(\sqrt{1+(\mathbf a_k^\top\mathbf u-y_k)^2/2}-1\Big).
\end{equation}
To additionally obtain a sparse regression model, we add an $\ell_1$-penalty term
\begin{equation}\label{eq:robustregression}
\min\limits_{\mathbf u\in\mathbb R^n} g(\mathbf u)+w \|\mathbf u\|_1,
\end{equation}
cf.\ \eqref{eq:min}, where the parameter $w>0$ acts as regularization parameter, see e.g.\ \cite{AlCrGe13,LiScRaCe15}. In the following, we assume that $\mathbf A=(\mathbf a_1\cdots\mathbf a_m)^\top\in\mathbb R^{m\times n}$ is injective. Then, the Hessian $\nabla^2 g(\mathbf u)$ is positive definite for all $\mathbf u\in\mathbb R^n$. However, it is not ensured that the level sets of $\Theta$ stay bounded. The data $\mathbf a_1,\ldots,\mathbf a_m\in\mathbb R^n$ are chosen normally distributed with standard deviation $1$ and mean $0$. We compute $\mathbf y^\delta=\mathbf A\mathbf u+\mathbf e$, where $\mathbf e \sim \mathcal N(0,1)$ and for a portion of the entries of $\mathbf y^\delta$ we choose $\mathbf e \sim \mathcal N(0,50)$, i.e.\ we construct outliers.

\begin{figure}[tbp]
\begin{center}
\includegraphics[scale=0.35]{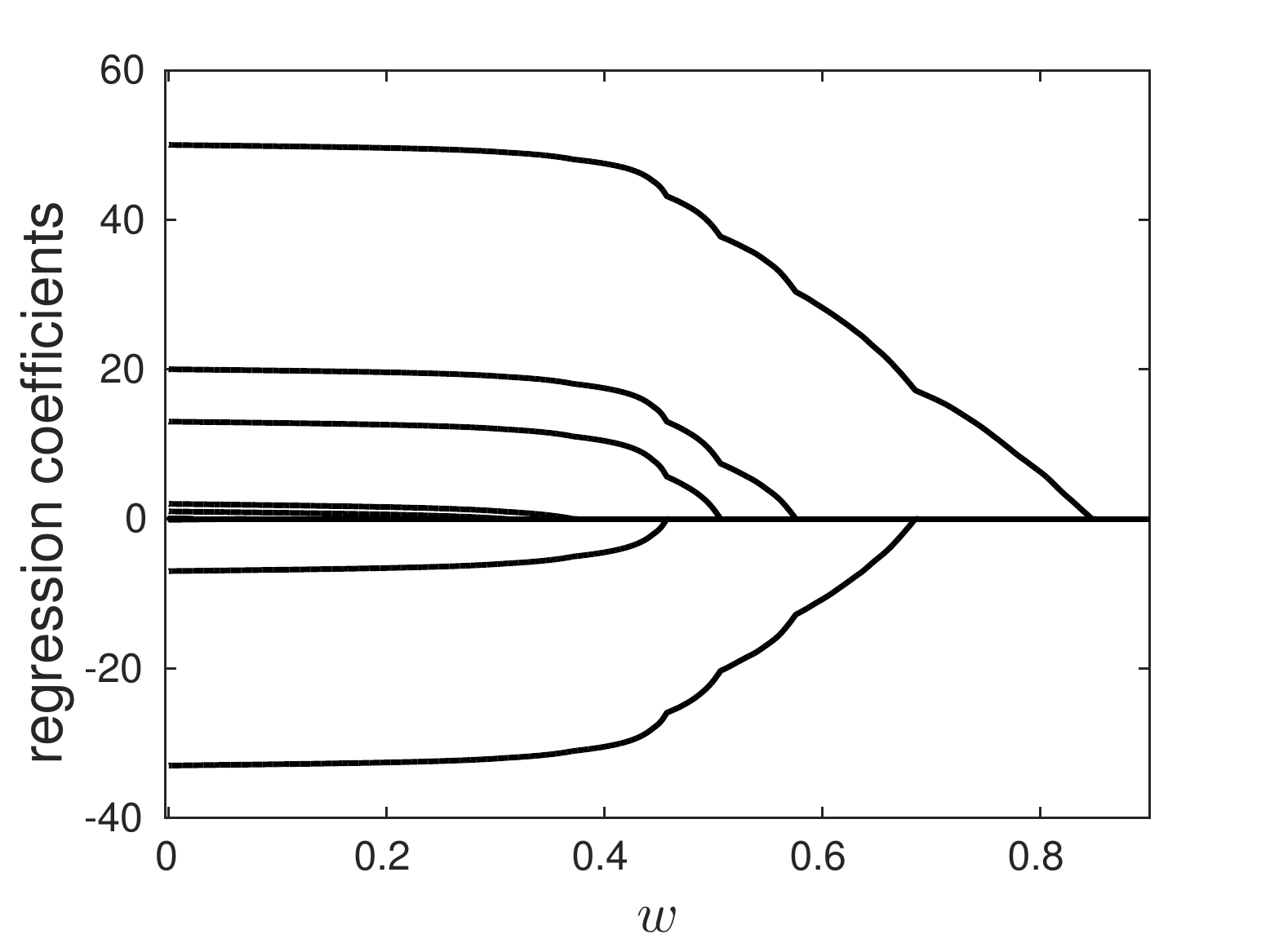}
\includegraphics[scale=0.35]{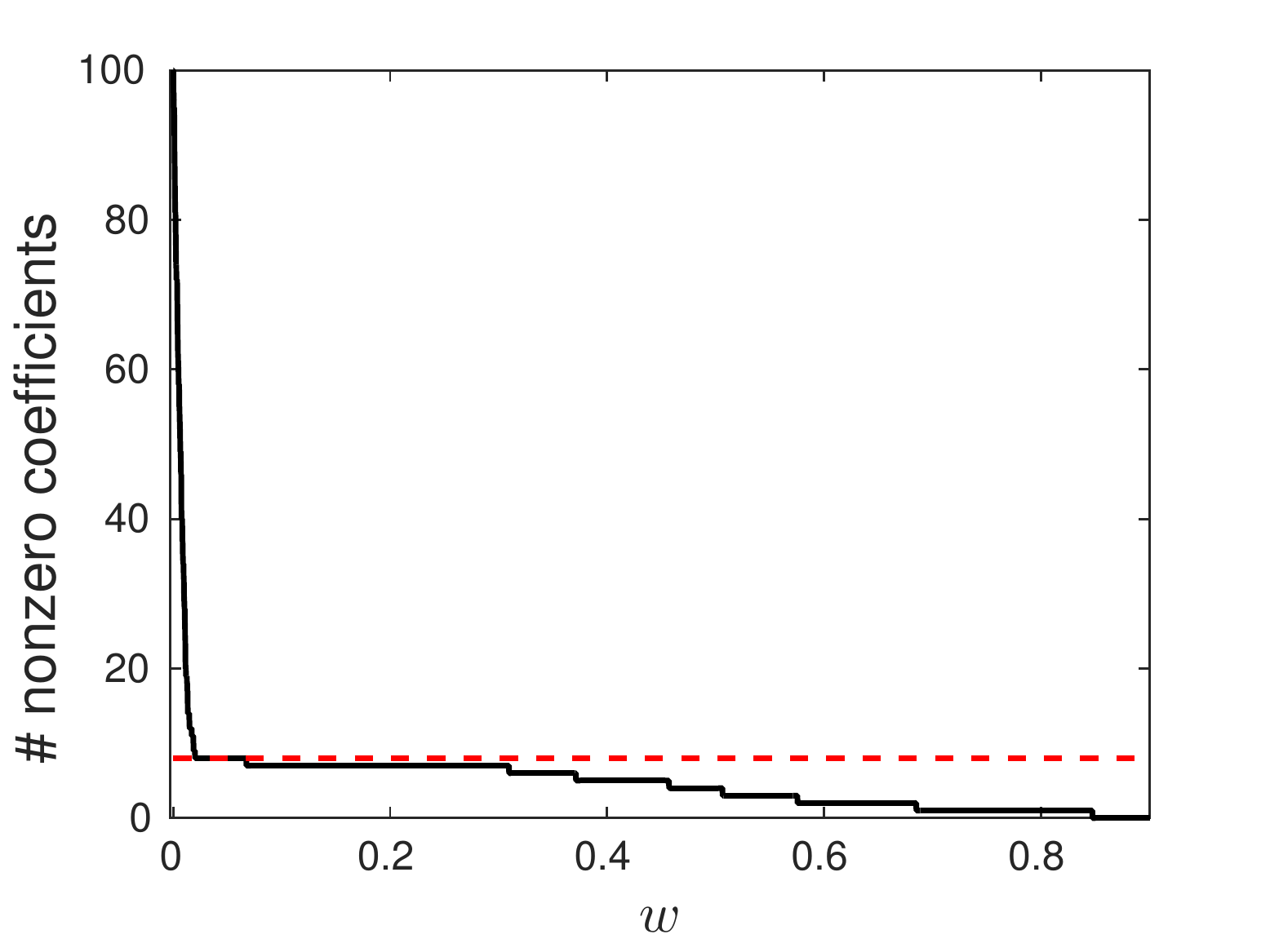}
\end{center}
\caption{Sparsity of the regression model depending on the choice of the regularization parameter $w$. Left: path of the computed regression coefficients. Right: number of nonzero regression coefficients (red dashed line: number of nonzero coefficients of the true regression model).}\label{fig:4}
\end{figure}

If the underlying model is unknown, there are several possibilities to select the regularization parameter $w$. For example, cross-validation may be used as proposed in \cite{AlCrGe13,Ti96}. Here, we assume that the true model is known. Similar to the parameter choice strategy proposed in \cite{KoLaNiSi12}, we choose the regularization parameter $w$ so that $\#\{k: (\mathbf u_w)_k\neq 0\}$ is equal to the number of nonzero elements of the true solution and $\mathbf u_w$ has minimal standard error
\begin{equation}\label{eq:sigma}
\sigma=\sqrt{\frac{1}{m-n-1}\sum_{k=1}^m (\mathbf a_k^\top\mathbf u_w-y_k)^2},
\end{equation}
 respectively maximal $R^2$-value 
\begin{equation}\label{eq:R2}
R^2=1-\frac{\frac{1}{m-n-1}\sum\limits_{k=1}^{m} (\mathbf a_k^\top\mathbf u_w-y_k)^2}{\frac{1}{m-1}\sum\limits_{k=1}^{m} (\overline y-y_k)^2},
\end{equation}
where $\mathbf u_w$ denotes the vector of computed regression coefficients for the regularization parameter $w$. Therefore, we minimize \eqref{eq:robustregression} for $w=\nu/10000$, $\nu=1,\ldots,9000$ and choose the starting vector $\mathbf u^{(0)}$ for $\nu>1$ as the solution to \eqref{eq:robustregression} of the last computation with $w=(\nu-1)/10000$. The true model is of the size $m=10000$, $n=100$ and has $8$ nonzero coefficients with weights $-33$, $-7$, $-0.1$, $1$, $2$, $13$, $20$ and $50$. The noisy vector $\mathbf y^\delta$ contains $10\%$ outliers. Figure \ref{fig:4} demonstrates the influence of the regularization parameter $w>0$ on the sparsity of the regression model. For the computations, we set $\gamma=10$ and the tolerance equal to $10^{-7}$ in Algorithm \texttt{modBSSN}. For very small $w$, all coefficients are chosen nonzero. If $w$ is chosen larger than $0.8474$, all coefficients are chosen equal to zero.

\begin{figure}[tbp]
\begin{center}
\includegraphics[scale=0.3]{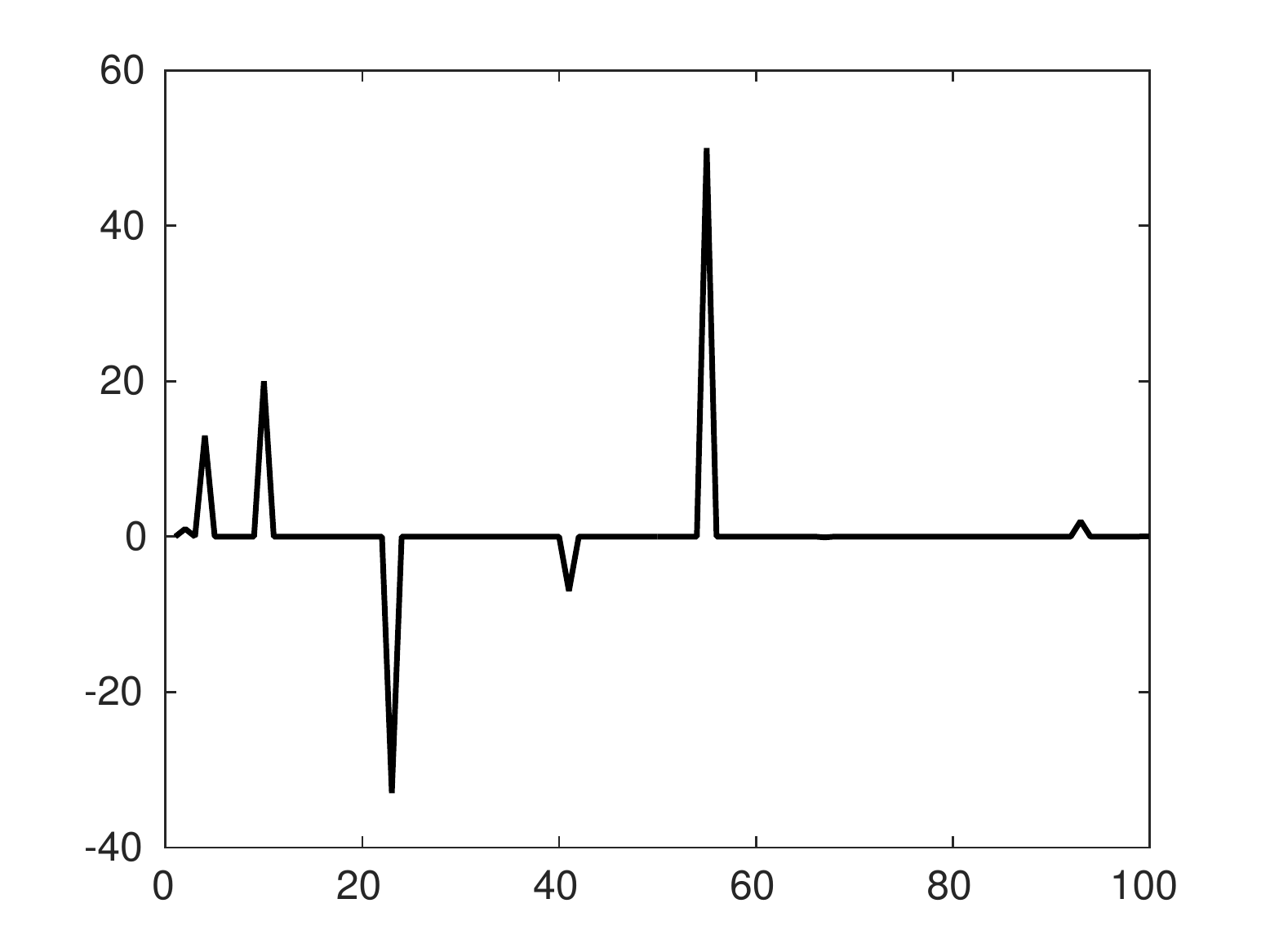}
\includegraphics[scale=0.3]{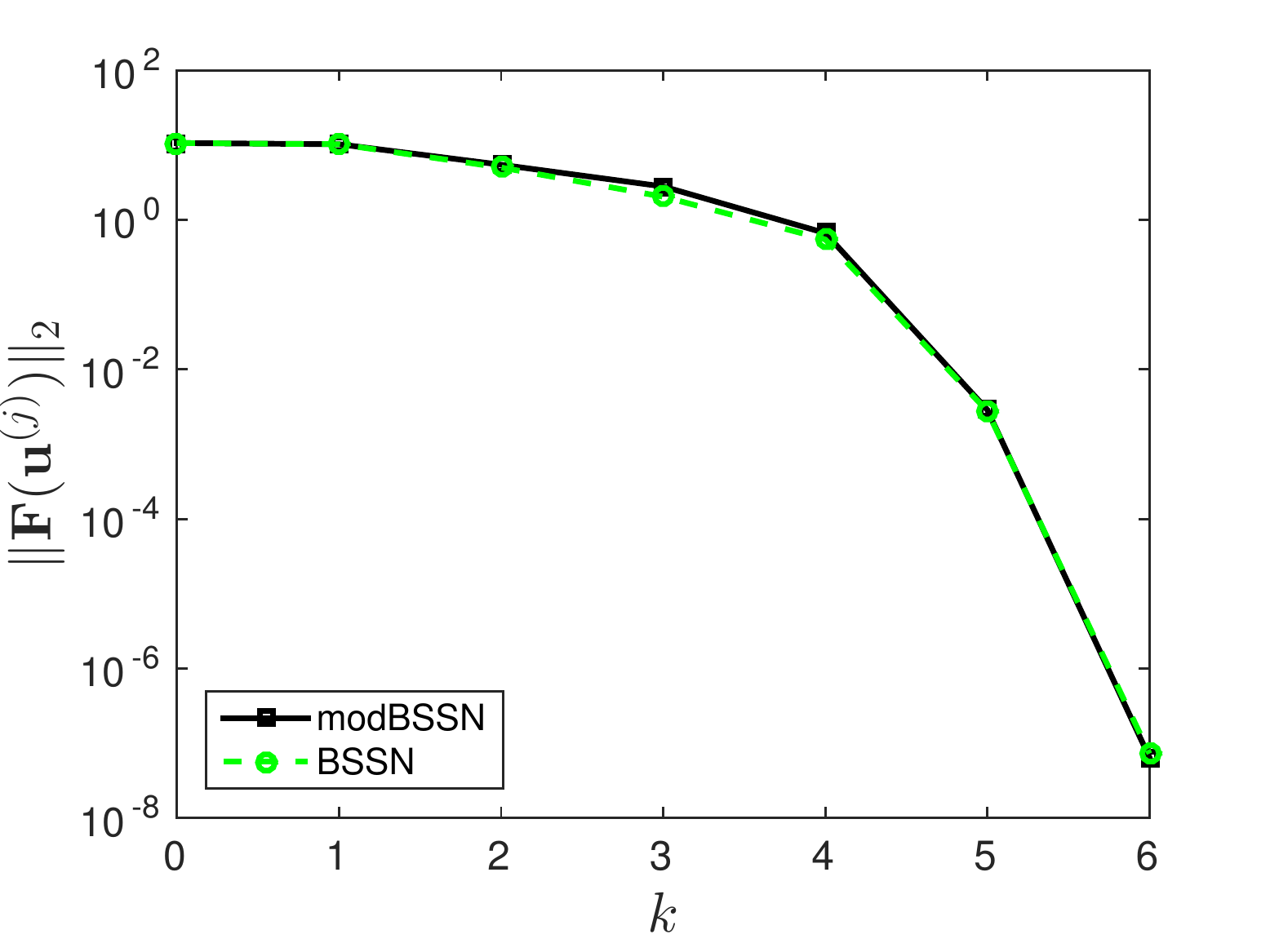}
\includegraphics[scale=0.3]{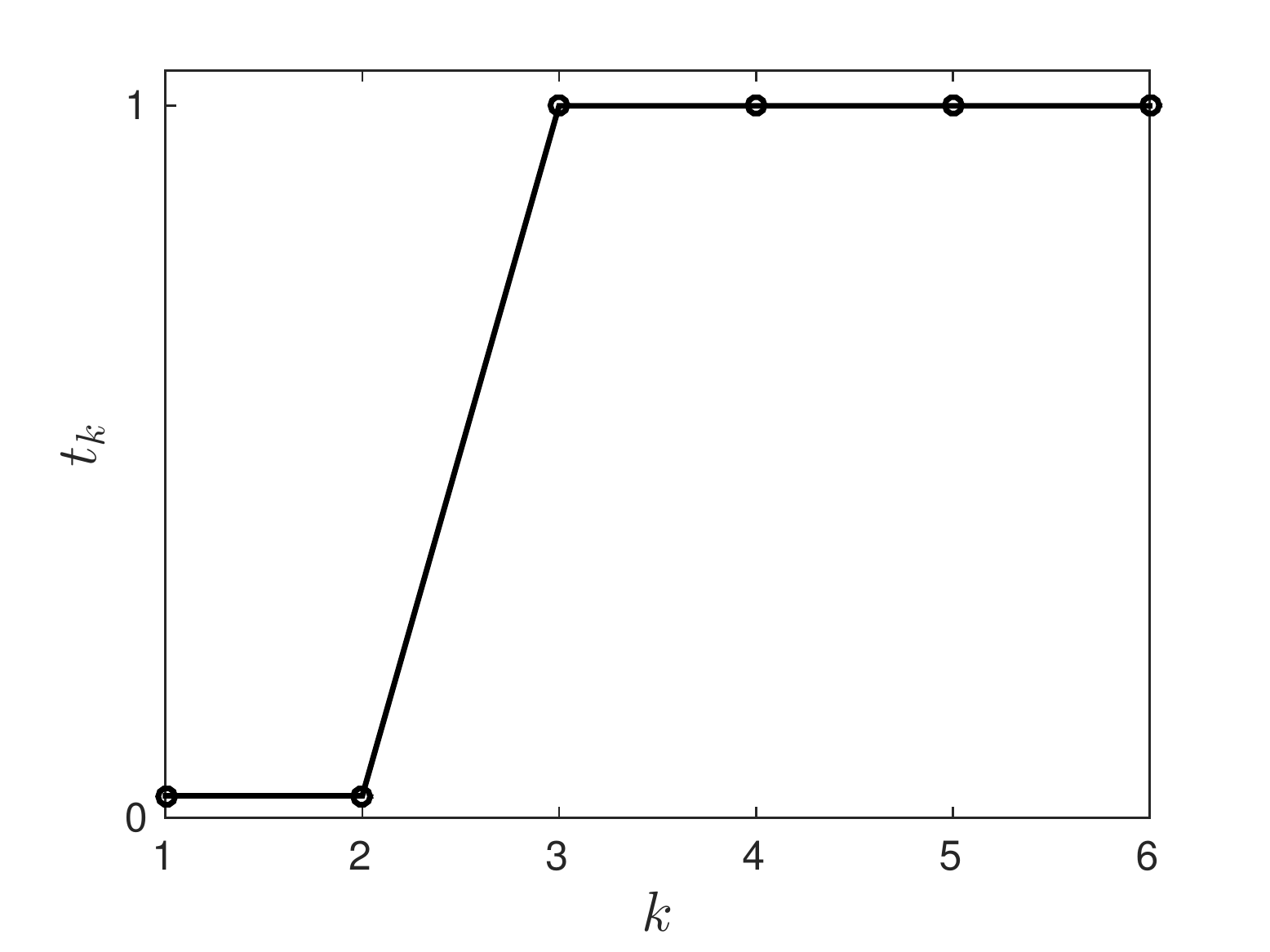}
\end{center}
\caption{Illustration of the performance of Algorithm $\mathtt{modBSSN}$ and Algorithm $\mathtt{BSSN}$ with $w=0.0201$ and $\gamma=10$ for the robust regression example. From left to right: true regression coefficients, residual norms, chosen stepsizes (identical for both algorithms).}\label{fig:5}
\end{figure}

Figure \ref{fig:5} shows the convergence properties of Algorithm \texttt{modBSSN} and the B-semismooth Newton method (\texttt{BSSN}) from Section \ref{sec:22} for the example from Figure \ref{fig:4}. We choose $w=0.0201$ and $\gamma=10$. Both algorithms converge within $6$ steps and the chosen stepsizes of the two algorithms coincide in this example. The stepsizes are four times chosen equal to $1$. For other values of $\gamma$, more Newton steps need to be computed. 

\section{Conclusion}
In the present paper, we are concerned with the efficient minimization of functionals of the type \eqref{eq:mininfinitedim}. In \cite{HaRa15}, a globalized B-semismooth Newton method was presented for quadratic discrepancy terms. Here, we generalized the me\-thod from \cite{HaRa15} to nonquadratic discrepancy terms. Additionally, by modifying index subsets, a modified algorithm was shown to be globally convergent without any additional requirements on the a priori unknown accumulation point of the sequence of iterates. Thus, we have overcome a theoretical drawback of \cite{HaRa15} concerning global convergence. Another advantage of the presented modified method is its local convergence speed. If an additional assumption is fulfilled, we have shown that the stepsizes are chosen eventually equal to $1$ and locally quadratic convergence is achieved.

By design, the proposed modified B-semismooth Newton method requires the solution of one linear complementarity problem per iteration, instead of one linear system as in other generalized Newton schemes. However, we have demonstrated that these systems stay small relative to the number of unknowns and therefore do not spoil the overall complexity. A hybrid version combines the efficiency of the B-semismooth Newton method and the convergence properties of the modified method.

In further research, one may focus on the development of globally convergent inexact Newton methods as proposed in \cite{EiWa94} for the smooth case enabling the design of matrix-free variants. Moreover, the globalization of quasi-Newton methods like \cite{MuHaMaPi13} could be considered.

\bibliography{arxiv_v3}
\bibliographystyle{abbrv}

\end{document}